\documentclass[12pt,reqno]{amsart}
\usepackage{amsmath}
\usepackage{amssymb}
\usepackage{amstext}
\usepackage{mathrsfs}
\usepackage{a4wide}
\usepackage{graphicx}
\usepackage{bbm}
\usepackage{verbatim}
\usepackage{bm}
\usepackage{graphicx}
\usepackage{tikz}
\usepackage[colorlinks,linkcolor=blue]{hyperref}
\usepackage{pgfplots}
\usepackage{tikz}
\usepackage {graphicx} 
\usepackage {epstopdf}
\usepackage{float}

\allowdisplaybreaks \numberwithin{equation}{section}
\usepackage{color}
\usepackage{cases}

\numberwithin{equation}{section}

\newtheorem{theorem}{Theorem}[section]

\newtheorem{lemma}[theorem]{Lemma}

\theoremstyle{definition}

\theoremstyle{remark}
\newtheorem{remark}[theorem]{Remark}

\newcommand{\ep}{\varepsilon}

\begin{document}

	\title
	[Regularization for point vortices on $\mathbb S^2$]{Regularization for point vortices on $\mathbb S^2$} 
	
	\author{Takashi Sakajo, Changjun Zou}
	
	\address{Department of Mathematics, Kyoto University, Kyoto, 606--8502, Japan}
	\email{sakajo@math.kyoto-u.ac.jp}
	\address{Department of Mathematics, Sichuan University, Chengdu, Sichuan, 610064, P.R. China}
	\email{zouchangjun@amss.ac.cn}
	
	\thanks{}

	\begin{abstract}
		We construct a series of patch type solutions for incompressible Euler equation on $\mathbb S^2$, which constitutes the regularization for steady or traveling point vortex systems. We first prove the existence of $k$-fold symmetric patch solutions, whose limit is the well-known von K\'arm\'an point vortex street on $\mathbb S^2$; then we consider the general steady case, where besides a non-localized part induced by the sphere rotation, $j$ positive and $k$ negative patches are located near a nondegenerate critical point of the Kirchhoff--Routh function 
on $\mathbb S^2$. Our construction is accomplished by Lyapunov--Schmidt reduction argument, where the traveling speed or vortex patch location are used to eliminate the degenerate direction of a linearized operator. 
We also show that the boundary of each vortex patch 
is a $C^1$ close curve, which is a perturbation of a small ellipse in the spherical coordinates. As far as we know, this is the first attempt for a regularization of the point-vortex equilibria on $\mathbb S^2$. 
	\end{abstract}
	
	\maketitle{\small{\bf Keywords:}  Incompressible Euler equation on $\mathbb S^2$; Regularization for point vortices; K\'arm\'an vortex street; Lyapunov--Schmidt reduction argument \\ 
		
		}
	
	\tableofcontents

	\section{Introduction}
	
	We construct a new kind of vorticity solution for the incompressible Euler equation on the rotating unit sphere $\mathbb S^2$ defined by
	$$\mathbb S^2:=\{ (x_1,x_2,x_3)\in\mathbb R^3 \mid x_1^2+x_2^2+x_3^3=1\},$$
	where the vorticity is constituted by $k$ positive and $j$ negative patches near a $k+j$ point vortex system. In this introduction, 
	we will derive the mathematical model, present some historical notes, expose our main results, and give the organization of the paper.

	\subsection{Euler equation on a unit sphere $\mathbb S^2$}
	
	Our research mainly deals with the vorticity solutions for the Euler equation on a unit sphere $\mathbb S^2$, which is a model to 
	describe the dynamic of westerlies or hurricanes on Earth, great red spots on Jupiter, and sunspots. These solutions are important 
	for meteorological predictions and the study of the motion of planets' atmospheres. For a more complete introduction to these equations, 
	we refer the readers to \cite{Gre, Yuri}. To derive the semilinear elliptic equation of the vorticity function and explain our 
	approach for construction, we first introduce some basic notions in mathematical analysis on $\mathbb S^2$, which is endowed 
	with a smooth manifold structure described by the following two charts:
	\begin{align*}
	   C_1: (0,\pi)\times(0,2\pi)&\to \mathbb R^3\\
	   (\theta,\varphi)&\to (\sin\theta\cos\varphi, \sin\theta\sin\varphi,\cos\theta)\\
	   C_2: (0,\pi)\times(0,2\pi)&\to \mathbb R^3\\
	   (\bar\theta,\bar\varphi)&\to (-\sin\bar\theta\cos\bar\varphi, -\sin\bar\theta,-\sin\bar\theta\sin\bar\varphi).
	\end{align*}
	In this paper, we mainly work on the first chart $C_1$, where the two variables $\theta$ and $\varphi$ are called the colatitude 
	and the longitude respectively. On the unit sphere $\mathbb S^2$, the Riemannian metric is given by
	\begin{equation*}
		\mathbf g_{\mathbb S^2}(\theta,\varphi)=d\theta^2+\sin^2\theta d\varphi^2.
	\end{equation*}
	Let $ \{N, S\}$ be the south and north poles of $\mathbb S^2$. Then for any point $\boldsymbol z = (\theta, \varphi) \in \mathbb S^2 \setminus \{N, S\}$ 
	on the sphere except for the two poles, an orthogonal
	basis of the tangent space $T_{\boldsymbol z}\mathbb S^2$ is given by
	\begin{equation*}
		\boldsymbol e_\theta=\partial_\theta \quad \mathrm{and} \quad \boldsymbol e_\varphi=\frac{\partial_\varphi}{\sin\theta},
	\end{equation*}
	where the classical identification between tangent vectors and directional differentiation is used. In this
	coordinate frame, the Riemannian volume is given by
	\begin{equation*}
		d\boldsymbol\sigma=\sin\theta d\theta d\varphi.
	\end{equation*}
	Denote $f(\boldsymbol z):\mathbb S^2\to \mathbb R$ as a function on $\mathbb S^2$. The integration for $f$ on the sphere can be written as 
	\begin{equation*}
		\int_{\mathbb S^2}f(\boldsymbol z)d\boldsymbol \sigma(\boldsymbol z)=\int_0^{2\pi}\int_0^{\pi}f(\theta,\varphi)\sin\theta d\theta d\varphi.
	\end{equation*}
	The gradient of $f$ is defined by
	\begin{equation*}
		\nabla_{\mathbb S^2} f(\theta,\varphi)=\partial_\theta f(\theta,\varphi)\boldsymbol e_\theta+\frac{\partial_\varphi f(\theta,\varphi)}{\sin\theta}\boldsymbol e_\varphi,
	\end{equation*}
	whose normal is 
	\begin{equation*}
		\nabla^{\perp}_{\mathbb{S}^2} f(\theta,\varphi)=J\nabla_{\mathbb{S}^2} f(\theta,\varphi),\quad\quad \mathrm{Mat}(J)=\left(
		\begin{array}{ccc}
			0 & 1 \\
			-1 & 0 \\
		\end{array}
		\right).
	\end{equation*}
	For a vector function $\boldsymbol f(\boldsymbol z) = (f_\theta(\theta, \varphi), f_\varphi(\theta, \varphi))$, the divergence is defined by
	\begin{equation*}
		\nabla_{\mathbb S^2} \cdot\boldsymbol f(\theta,\varphi)=\frac{1}{\sin\theta}\partial_\theta (\sin\theta f_\theta(\theta,\varphi)) +\frac{1}{\sin\theta}\partial_\varphi f_\varphi(\theta,\varphi).
	\end{equation*}
	Using the gradient and the divergence on the sphere, we can write down the Laplace--Beltrami operator on  $\mathbb S^2$, 
	\begin{equation*}
		\Delta_{\mathbb S^2} f(\theta,\varphi)=\frac{1}{\sin\theta}\partial_\theta (\sin\theta\partial_\theta f(\theta,\varphi)) +\frac{1}{\sin^2\theta}\partial^2_\varphi f(\theta,\varphi),
	\end{equation*}
	where the north and south poles are two singular points.
	In the sequel, we use the Green representation for $-\Delta_{\mathbb S^2}$, namely, 
	\begin{equation*}
	   (-\Delta)^{-1}_{\mathbb S^2}f=\int_{\mathbb S^2}G(\boldsymbol z,\boldsymbol z')d\boldsymbol\sigma(\boldsymbol z'),
	\end{equation*}
	as an important tool for the construction of localized vorticity solutions, where the function $G(\boldsymbol z,\boldsymbol z')$ takes the form
	\begin{equation*}
		G(\theta,\varphi,\theta',\varphi')=-\frac{1}{4\pi}\ln\big(1-\cos\theta\cos\theta'-\sin\theta\sin\theta'\cos(\varphi-\varphi')\big)+\frac{\ln2}{4\pi}.
	\end{equation*}
	Notice that  
	\begin{align*}
		D(\theta,\phi,\theta',\varphi')&=1-\cos\theta\cos\theta'-\sin\theta\sin\theta'\cos(\varphi-\varphi')\\
		&=2\left[\sin^2\left(\frac{\theta-\theta'}{2}\right)+\sin\theta\sin\theta'\sin^2\left(\frac{\varphi-\varphi'}{2}\right)\right].
	\end{align*}
    By the Taylor expansion for the sine function, we see that the singular part in $G(\boldsymbol z,\boldsymbol z')$ is similar to the fundamental solution $-\frac{1}{2\pi}\ln\frac{1}{|\boldsymbol x|}$ for $-\Delta$ in $\mathbb R^2$, which is defined as
    \begin{equation*}
    	\Gamma(\theta,\varphi,\theta',\varphi')=-\frac{1}{4\pi}\ln\big[(\theta-\theta')^2+(\varphi-\varphi')^2 \sin^2\theta\big].
    \end{equation*}
    Then we can split $G(\boldsymbol z,\boldsymbol z')$ into two parts
    \begin{equation}\label{1-1}
    	\begin{split}
    	G(\theta,\phi,\theta',\phi')&=\Gamma(\theta,\varphi,\theta',\varphi')+\left[G-\Gamma(\theta,\varphi,\theta',\varphi')\right]\\
    	&=\Gamma(\theta,\varphi,\theta',\varphi')+H(\theta,\varphi,\theta',\varphi').
    	\end{split}
    \end{equation}
    The function $H(\theta,\varphi,\theta',\varphi')\in C^1(B_\delta(\boldsymbol z)\times B_\delta(\boldsymbol z))$ for any $\boldsymbol z\in \mathbb S^2\setminus \{N, S\}$ is the remaining regular part, 
    in which $B_\delta(\boldsymbol z)$ denotes a spherical cap around $\boldsymbol z$ with a small radius $\delta$.
    This observation is a starting point for our construction of solutions close to point vortices. 
	
	Now we introduce the vorticity formulation for the incompressible Euler equations on a rotating sphere $\mathbb S^2$, which takes the form
	\begin{align}\label{1-2}
		\begin{cases}
			\partial_t\omega+\boldsymbol{v}\cdot \nabla_{\mathbb S^2} (\omega-2\gamma\cos\theta)=0,\\
			\ \boldsymbol{v}=\nabla^\perp_{\mathbb S^2}(-\Delta)^{-1}_{\mathbb S^2}\omega,\\
			\omega\big|_{t=0}=\omega_0,
		\end{cases}
	\end{align}
	where $\omega$ is the vorticity function, $\boldsymbol v$ is the velocity field, and $\gamma$ is the uniform angular rotation speed. In \eqref{1-2}, the term 
	$-2\gamma\boldsymbol v\cdot \nabla_{\mathbb S^2}\cos\theta$ corresponds to the Coriolis force coming from the rotation of the sphere, 
	and the second equation is the Biot--Savart law on $\mathbb S^2$, which means the velocity field $\boldsymbol v$ is divergence free.
	According to the divergence theorem, we can deduce the Gauss constraint 
	\begin{equation*}
		\int_{\mathbb S^2}\omega(t,\boldsymbol z)d\boldsymbol \sigma(\boldsymbol z)=0,\quad \forall \, t\in[0,+\infty).
	\end{equation*}
	
    In this paper, we first consider the traveling wave solutions as a special relative equilibrium of \eqref{1-2} with $\gamma=0$, 
    which is obtained by a translation acting on initial data $\omega_0$, namely, 
	\begin{equation*}
		\omega(t,\theta,\varphi)=\omega_0(\theta, \varphi+Wt).
	\end{equation*}
	By substituting it into \eqref{1-2} and letting $\psi=(-\Delta)^{-1}_{\mathbb S^2}\omega$ be the stream function, we find that the first equation of \eqref{1-2} becomes
	\begin{equation*}
		\nabla_{\mathbb{S}^2}^\perp(\psi+W\cos\theta)\cdot \nabla_{\mathbb S^2} \omega=0,
	\end{equation*}
	which means that $\psi+W\cos\theta$ and $\omega$ are functional dependent. Hence if we impose
	\begin{equation}\label{1-3}
		(-\Delta_{\mathbb S^2})\psi=F(\psi+W\cos\theta),
	\end{equation}
    then $\omega=(-\Delta)_{\mathbb{S}^2}\psi$ gives a traveling wave solution to \eqref{1-2} with $\gamma=0$. It should be noticed that if $\gamma\neq 0$, then the support of potential vorticity $\omega-2\gamma\cos\theta$ is the whole sphere $\mathbb S^2$, which will bring extraordinary difficulties for the construction of traveling localized vortices $\omega$, even if for point vortex systems. For discussions on this direction, we refer interested readers to \cite{New0}. 
    
    However, we can construct non-localized steady solutions to \eqref{1-2} with $\gamma\neq0$ such that $\partial_t\omega=0$. In this situation, the first equation in \eqref{1-2} turns to be
    \begin{equation*}
    	\nabla_{\mathbb{S}^2}^\perp\psi\cdot \nabla_{\mathbb S^2} (\omega-2\gamma\cos\theta)=0,
    \end{equation*}
    from which we deduce that if $\psi$ satisfies 
    	\begin{equation}\label{1-4}
    	(-\Delta_{\mathbb S^2})\psi=F(\psi)+2\gamma\cos\theta,
    \end{equation}
    then $\omega=(-\Delta)_{\mathbb{S}^2}\psi$ gives a steady solution to \eqref{1-2} on a rotating sphere at speed $\gamma$. 
    
    In fact, these two kind solutions can be transformed to each other. Notice that it holds $(-\Delta_{\mathbb S^2})\cos\theta=2\cos\theta$. We have following lemma, which describes the equivalence for \eqref{1-3} and \eqref{1-4}.
    \begin{lemma}\label{lem1-1}
    	Let $\psi$ be a solution of
    	$$(-\Delta_{\mathbb S^2})\psi=F(\psi+W\cos\theta).$$
    	Then $\psi_\gamma=\psi+\gamma\cos\theta$ gives a solution to
    	$$(-\Delta_{\mathbb S^2})\psi_\gamma=F\big(\psi_\gamma+(W-\gamma)\cos\theta\big)+2\gamma\cos\theta.$$
    \end{lemma} 
    In view of Lemma \ref{lem1-1}, once we construct a traveling localized wave solution to \eqref{1-3}, we actually obtain a class of non-localized solutions to \eqref{1-4} with different sphere rotation speed.

    \subsection{Historical notes}
    
    The vorticity formulation of incompressible Euler equation in $\mathbb R^2$ reads
    \begin{align}\label{1-5}
    	\begin{cases}
    		\partial_t\omega+\boldsymbol{v}\cdot \nabla \omega=0,\\
    		\ \boldsymbol{v}=\nabla^\perp(-\Delta)^{-1}\omega,\\
    		\omega\big|_{t=0}=\omega_0.
    	\end{cases}
    \end{align}
    The global well-posedness for \eqref{1-5} with initial data in $L^1\cap L^\infty$ was established by Yudovich \cite{Yud} in 1963, 
    which is also extended to the Euler equation on the sphere \eqref{1-2}. 
    Since then, the theory of weak solutions has been considerably improved, see \cite{Del, Dip, Elg, MB}.

    Besides the well-posedness, researchers who study fluid mechanics are also interested in specific global solutions of \eqref{1-2} or \eqref{1-5}, 
    which maintain their shape during evolution. The general approach for giving such solutions is to make perturbations near trivial ones, 
    which can be roughly divided into two large classes. 
    
    Owing to the structure of the nonlinear term, it is well-known that all radial functions are stationary solutions to \eqref{1-5}, 
    and the vortices bifurcating from these trivial ones constitute the first class of global solutions. These $N$-fold symmetric solutions rotate at a uniform angular velocity, 
    which is also known as V-states. In 1978, Deem and Zabusky \cite{Deem} used numerical simulations to show the possible existence of V-states. 
    A few years later, Burbea \cite{Burb} put forward a new method of finding V-states for \eqref{1-5}, where local bifurcation was applied in Hardy space, 
    and the key observation is to rewrite \eqref{1-5} in a functional analytic framework by conformal mapping. 
    However, the proof in \cite{Burb} has a gap, since the functions in Hardy space do not actually have derivatives. It was not until 2013 
    that Hmidi, Mateu, and Verdera \cite{Hmi} gave the first rigorous proof for the existence of simply--connected V-states by contour dynamics equations, 
    and Hassainia and Hmidi \cite{Has} extended this approach to the gSQG case. Since then, the bifurcation method has been widely used in the construction
    of V-states for active scalar equations. For instance, De La Hoze, Hassaina, Hmidi, Mateu, Verdera \cite{de2} proved the existence of doubly connected 
    vortex patches bifurcated from annuli at specific angular velocity. While $2$-fold symmetric V-states near Kirchhoff elliptic vortices in Euler flow \cite{Kir} 
    and V-states with smooth vorticity distribution were discovered by Castro, C\'ordoba, and G\'omez-Serrano in \cite{Cas3}. 
    Moreover, Hassainia, Masmoudi, and Wheeler \cite{Has0} made continuation on the bifurcation parameter by analytic tool, and studied the global 
    behavior of solutions. For quasi-periodic cases, there are also some results by KAM theory, see \cite{BHM, HHM}.
    
    For the equation \eqref{1-2} on $\mathbb S^2$, the analogies for V-States of \eqref{1-5} in $\mathbb R^2$ are called zonal solutions, 
    where the stream function $\psi(\theta,\varphi)=\psi(\theta)$ is longitude independent. 
    Especially, The zonal Rossby--Haurwitz stream functions of degree $n\in\mathbb N$ 
    $$\Psi_n(\theta)=\beta Y_n^0(\theta)+\frac{2\gamma}{n(n+1)-2}\cos\theta, \quad \beta\in\mathbb R_+$$
    are special stationary solutions, where $Y_n^0(\theta)$ are spherical harmonic functions. 
    In \cite{CG}, Constantin and Germain considered local and global bifurcation of non-zonal solutions to \eqref{1-2} from Rossby--Haurwitz waves. 
    They also proved the stability in $H^2(\mathbb S^2)$ of degree $2$ case as well as the instability in $H^2(\mathbb S^2)$ of 
    more general non-zonal Rossby--Haurwitz type solutions. They also showed that any solution to \eqref{1-2} with $F'>-6$ must be zonal (modulo rotation)
    and stable in $H^2(\mathbb S^2)$ provided an additional constraint $F' < 0$, where the constant $-6$ corresponds to the second eigenvalues 
    of the Laplace--Beltrami operator. Very recently, Garc\'ia, Hassainia, and Roulley \cite{GHR} developed the idea in \cite{Hmi}, and constructed $k$-fold
     symmetric vortex cap solutions by using contour dynamic equations and bifurcation tool. For other existence and stability results on this aspect, 
     we refer to \cite{CWZ, Cap, Nua}.
    
    In this sequel, we mainly consider the second class of global solutions, which consist of relatively regular vortex patches 
    near a system of point vortices given by $$\omega^*(\boldsymbol x)=\sum_{n=1}^N\kappa_n\boldsymbol \delta_{\boldsymbol x_n}$$
    with $\boldsymbol \delta_{\boldsymbol x}$ the Dirac measure at point $\boldsymbol x$. Although the velocity field at the location of each point vortex is singular,
    their dynamics can be understood in a way that each point vortex is transported by the velocity field created by the other point vortices. 
    More precisely, this point vortex system on a two-dimensional domain $D$ is dominated by a Hamiltonian system, 
    \begin{equation*}
    	\kappa_n\boldsymbol \dot {\boldsymbol x}_n=\nabla^{\perp}_{\boldsymbol x_n}\mathcal K_N, \quad 1\le n\le N,
    \end{equation*}
    where $\mathcal K_N$ is called the Kirchhoff--Routh function defined as
    \begin{equation*}
    	\mathcal K_N(\boldsymbol x_1,\boldsymbol x_2,..,\boldsymbol x_N):=\frac{1}{2}\sum_{m,n=1,m\neq n}^N\kappa_m\kappa_nG_D(\boldsymbol x_n,\boldsymbol x_m)+\frac{1}{2}\sum_{m=1}^N\kappa_m^2H_D(\boldsymbol x_m,\boldsymbol x_m)
    \end{equation*}
    with $G_D$ denotes the Green kernel on $D$, $H_D$ is its corresponding Robin function. See \cite{Lin1, Lin2} for the planar case 
    and \cite{Drit} for the spherical case. Owing to the Hamilton structure for the motion of point vortices, the point-vortex dynamics 
    on the Riemann surface has been discussed by many mathematicians, and different kinds of relative equilibrium were found, 
    for which we refer to \cite{New0, New, Sakajo-Torus, Wang}.
    
    The procedure of approximating $\omega^*$ by relatively regular solutions is called a regularization or a desingularization of point vortices. 
    For the planar case \eqref{1-4}, there is abundant literature on this field. In \cite{March}, Marchioro and Pulvirenti considered the evolution 
    of vorticity in the Euler flow with smooth initial data near $k$ point vortices. 
    The same topic was discussed by Davila, Del Pino, Musso, and Wei via a newly developed gluing method in \cite{Dav}. 
    As for the regularization for equilibria in the Euler flow, Turkington constructed global vortex patch solutions near point vortices in \cite{T1}, 
    which is inspired by Arnold's dual variational principle \cite{Ar2}, and is now known as the vorticity method. 
    There are also other approaches towards the same aim, which focus on the stream function and its related semilinear elliptic equation. 
    For example, in \cite{SV}, Smets and Van Schaftingen gave a regularization of stationary vortices in different cases by the mountain-pass lemma; 
    in \cite{Cao3}, Cao, Peng, and Yan considered the vortex patch problem of the planer Euler equation by singular perturbation of stream function. 
    The contour dynamic equation used for the construction of V-states is also available for a regularization procedure, which we refer to \cite{HM} 
    for the construction of co-rotating and counter-rotating patch pairs. 
    
    In the following, we will construct a series of patch solutions to approximate the von K\'arm\'an vortex street on $\mathbb S^2$, and then deal 
    with a general case. Our construction is more difficult and delicate than the planar case due to the different structure of Laplacian. 
    However, we will show that the dynamic property of small vortex patches on $\mathbb S^2$ is strongly dependent on the corresponding 
    Kirchhoff--Routh function, which is similar to the regularization for point vortices on $\mathbb R^2$.
  
    \subsection{The main results}
    
    Let us start our research from the most classic point vortex system on the unit sphere $\mathbb S^2$, namely, the $k$ periodic von K\'arm\'an vortex street, 
    where there are $2k$ point vortices with the circulation $\kappa$, half positive and half negative, traveling at a uniform angular speed along the equator. 
    
    To be more precise, for $1\le i\le k$, let $\boldsymbol z_i^+=(\theta_0,\varphi_i^+)$ be the location of positive vortices on the northern hemisphere and $\boldsymbol z_i^-=(\pi-\theta_0,\varphi_i^-)$ be the location of negative vortices on the southern hemisphere. There are two types of vortex streets, the first type is 
    $$\varphi_i^+=\frac{2\pi i}{k}-\frac{\pi}{k}, \quad \quad \varphi_i^-=\frac{2\pi i}{k}-\frac{\pi}{k},\quad\quad\quad\ \ \mathbf{(type\, 1)} $$
    and the second type is
    $$\varphi_i^+=\frac{2\pi i}{k}-\frac{3\pi}{2k}, \quad \quad \varphi_i^-=\frac{2\pi i}{k}-\frac{\pi}{2k}.\quad\quad\quad \mathbf{(type\, 2)} $$
    Then the point vortex solution can be written as
    $$\omega^*(\boldsymbol z)=\kappa\sum\limits_{i=1}^{k}\boldsymbol \delta_{\boldsymbol z_i^+}-\kappa\sum\limits_{i=1}^{k}\boldsymbol \delta_{\boldsymbol z_i^-},$$
    whose stream function is 
    $$\psi^*(\boldsymbol z)=\kappa\sum\limits_{i=1}^{k}G(\boldsymbol z,\boldsymbol z_i^+)-\kappa\sum\limits_{i=1}^{k}G(\boldsymbol z,\boldsymbol z_i^-).$$    
    The point vortex model constitutes a Hamilton system, and the uniform traveling angular velocity of the vortex street is given by
    \begin{equation*}
    	W^*=\frac{\kappa}{\sin\theta_0}\sum\limits_{i=2}^k\partial_\theta G(\boldsymbol z_1^+,\boldsymbol z_i^+)+\frac{\kappa}{\sin\theta_0}\partial_\theta H(\boldsymbol z_1^+,\boldsymbol z_1^+)-\frac{\kappa}{\sin\theta_0}\sum\limits_{i=1}^k\partial_\theta G(\boldsymbol z_1^+,\boldsymbol z_i^-).
    \end{equation*}  
    See \cite{Drit}. Note that the vortex shedding can be put in a stationary frame with sphere rotation speed $\gamma=W^*$.
    
    We construct a series of vorticity solutions $\{\omega_\varepsilon\}$, tending to $\omega^*$ but with a more regular form than the Dirac measures.
    This procedure is hence regarded as a regularization for point vortices. In \cite{CQZZ}, the authors have considered the $C^1$ regular solutions 
    for the whole plane $\mathbb R^2$. However, we will consider the solution of patch type in this paper, whose initial data is the characteristic 
    function of the $2k$ domain, 
    \begin{equation*}
    	\omega_\varepsilon(\boldsymbol z)=\frac{1}{\varepsilon^2}\sum\limits_{i=1}^k\boldsymbol 1_{\Omega_{i,\varepsilon}^+}-\frac{1}{\varepsilon^2}\sum\limits_{i=1}^k\boldsymbol 1_{\Omega_{i,\varepsilon}^-}
    \end{equation*}
    with $\Omega_{i,\varepsilon}^\pm$ are small areas centered at $\boldsymbol z_i^\pm$,  and $\varepsilon$ is a small scale parameter. 
    As $\varepsilon\to 0^+$, $\omega_*$  tends to the point vortex solution $\omega_*$. 
    Actually, the stream function $\psi_\varepsilon=(-\Delta_{\mathbb S^2})^{-1}\omega_\varepsilon$ satisfies the following equation
    \begin{equation}\label{1-6}
    	(-\Delta_{\mathbb S^2})\psi_\varepsilon=\frac{1}{\varepsilon^2}\boldsymbol1_{\{\psi_\varepsilon+W_\varepsilon\cos\theta>\mu_\varepsilon\}}-\frac{1}{\varepsilon^2}\boldsymbol1_{\{-\psi_\varepsilon-W_\varepsilon\cos\theta>\mu_\varepsilon\}}
    \end{equation}
    by \eqref{1-3}, where 
    $$\{\psi_\varepsilon+W_\varepsilon\cos\theta>\mu_\varepsilon\}=\bigcup_{i=1}^k\Omega_{i,\varepsilon}^+, \quad\quad \{-\psi_\varepsilon-W_\varepsilon\cos\theta>\mu_\varepsilon\}=\bigcup_{i=1}^k\Omega_{i,\varepsilon}^-$$ 
   are the positive and negative vorticity sets separately, $W_\varepsilon$ is the traveling angular velocity, and $\mu_\varepsilon$ is a flux constant. 
   
   To simplify the problem, we assume $\psi_\varepsilon$ and $\omega_\varepsilon$ are of $k$-fold symmetry. That is to say, for any $1\leq i \leq k$,
   $$\psi_\varepsilon\left(\theta+\frac{2\pi i}{k},\varphi\right)=\psi_\varepsilon(\theta,\varphi) \quad \mathrm{and} \quad \omega_\varepsilon\left(\theta+\frac{2\pi i}{k},\varphi\right)=\omega_\varepsilon(\theta,\varphi).$$
   Moreover, we require that $\Omega_{i,\varepsilon}^+$ and $\Omega_{i,\varepsilon}^-$ are symmetric with respect to $(\frac{\pi}{2},\frac{2\pi i}{k}-\frac{\pi}{k})$, 
   so that they have the same shape. Then we can only deal with $\Omega_{1,\varepsilon}^+$ in the construction. Having done these preparations, 
   we have the main theorem in our paper, which is stated as follows.
    \begin{theorem}\label{thm1}
   	Suppose $\boldsymbol z_i^\pm$ satisfy $\mathbf{(type\, 1)}$ or $\mathbf{(type\, 2)}$ with $0<\theta_0\le\pi/2$. Then there exists an $\varepsilon_0>0$ small, such that for each $\varepsilon\in (0,\varepsilon_0]$, \eqref{1-6} has a solution $(\psi_\varepsilon, W_\varepsilon)$. Let $\omega_\varepsilon=(-\Delta_{\mathbb S^2})\psi_\varepsilon$ be the vorticity function. We have the following asymptotic estimates:
   	\begin{itemize}
   		\item[(i)] One has
   		$$\omega_\varepsilon\rightharpoonup \kappa\sum\limits_{i=1}^{k}\boldsymbol \delta_{\boldsymbol z_i^+}-\kappa\sum\limits_{i=1}^{k}\boldsymbol \delta_{\boldsymbol z_i^-} \quad as \ \varepsilon\to 0^+,$$
   		where the convergence is in the sense of measures, and $\kappa$ is the circulation for each vortex.
   		\item[(ii)] The boundary of vorticity set $\Omega_{i,\varepsilon}^\pm$ is a $C^1$ close curve, which is parameterized as
   		$$\partial \Omega_{i,\varepsilon}^\pm=\left\{\boldsymbol z_i^\pm+\big[\sqrt{\kappa/\pi}\varepsilon+o(\varepsilon)\big](\cos\xi,\sin^{-1}\theta_0\sin\xi) \mid \xi\in[0, 2\pi) \right\}.$$
   		\item[(iii)] As $\varepsilon\to 0^+$, the traveling angular velocities are given by
   		\begin{align*}
   			W_\varepsilon^{(1)}\to&\frac{\kappa}{\sin\theta_0}\sum\limits_{i=2}^k\partial_\theta G\left(\theta_0,\frac{\pi}{k},\theta_0,\frac{2\pi i}{k}-\frac{\pi}{k}\right)+\frac{\kappa}{\sin\theta_0}\partial_\theta H\left(\theta_0,\frac{\pi}{k},\theta_0,\frac{\pi }{k}\right)\\
   			&-\frac{\kappa}{\sin\theta_0}\sum\limits_{i=1}^k\partial_\theta G\left(\theta_0,\frac{\pi}{k},\pi-\theta_0,\frac{2\pi i}{k}-\frac{\pi}{k}\right)
   		\end{align*}
   		for $(\mathbf{type\, 1})$ vortex street, and
   		\begin{align*}
   			W_\varepsilon^{(2)}\to&\frac{\kappa}{\sin\theta_0}\sum\limits_{i=2}^k\partial_\theta G\left(\theta_0,\frac{\pi}{2k},\theta_0,\frac{2\pi i}{k}-\frac{3\pi}{2k}\right)+\frac{\kappa}{\sin\theta_0}\partial_\theta H\left(\theta_0,\frac{\pi}{2k},\theta_0,\frac{\pi i}{2k}\right)\\
   			&-\frac{\kappa}{\sin\theta_0}\sum\limits_{i=1}^k\partial_\theta G\left(\theta_0,\frac{\pi}{2k},\pi-\theta_0,\frac{2\pi i}{k}-\frac{\pi}{2k}\right)
   		\end{align*}
   	    for $(\mathbf{type\, 2})$ vortex street.
   	\end{itemize}
   \end{theorem}
   
   \begin{remark}
   	With the existence result in Theorem \ref{thm1}, we actually obtain a class of non-localized solutions to \eqref{1-2} by Lemma \ref{lem1-1}. If we let $\tilde\omega_\ep=\omega_\varepsilon+2\gamma\cos\theta$ and $\tilde{\boldsymbol v}=\nabla^\perp_{\mathbb S^2}\psi_\varepsilon+(0,\gamma\sin\theta)$, then we give a traveling wave solution at the angular velocity $W_\gamma=W-\gamma$ on a sphere with $\gamma$ the rotation speed.
   \end{remark}
   
   The von K\'arm\'an point vortex street is only a special case of steady or traveling point vortex system on $\mathbb S^2$. For other vivid examples, 
   we refer to \cite{New}, where different kinds of point vortex equilibria are given. To get a step further for general cases, we consider \eqref{1-4}, and let 
   \begin{equation*}
   	  \omega^*(\boldsymbol z)=\sum\limits_{m=1}^{j}\kappa_m^+\boldsymbol \delta_{\boldsymbol z_m^+}-\sum\limits_{n=1}^{k}\kappa_n^-\boldsymbol \delta_{\boldsymbol z_n^-}+2\gamma\cos\theta
   \end{equation*}
   be the limit vortex function. In this situation, $\omega^*(\boldsymbol z)$ is composed of singular point vortex part and regular non-localized part $2\gamma\cos\theta$, which is no longer a point vortex system but a vortex-wave system. However, in view of Lemma \ref{lem1-1}, it can be regarded as equal to a point vortex system
   \begin{equation*}
   	\hat\omega^*(\boldsymbol z)=\sum\limits_{m=1}^{j}\kappa_m^+\boldsymbol \delta_{\boldsymbol z_m^+}-\sum\limits_{n=1}^{k}\kappa_n^-\boldsymbol \delta_{\boldsymbol z_n^-}
   \end{equation*}
   with traveling speed $W=\gamma$ on a stationary sphere. 
   
   For $\omega^*(\boldsymbol z)$, the Gauss constraint implies that
   $$\sum\limits_{m=1}^{j}\kappa_m^+=\sum\limits_{n=1}^{k}\kappa_n^-.$$
   By the dynamics of point vortices on $\mathbb S^2$ \cite{Drit}, the coordinates $(\boldsymbol z_1^+,\cdots,\boldsymbol z_j^+, \boldsymbol z_1^-,\cdots, \boldsymbol z_k^-)$ with $\boldsymbol z_m^+=(\phi_m^+,\theta_m^+)$, $\boldsymbol z_n^-=(\phi_n^-,\theta_n^-)$ ($\theta_m^+,\theta_n^-\neq 0,\pi$) should be on a critical point of the Kirchhoff--Routh function
   \begin{equation}\label{1-7}
   	\begin{split}
   	\mathcal K_{k+j}&(\boldsymbol z_1^+,\cdots,\boldsymbol z_j^+,\boldsymbol z_1^-,\cdots,\boldsymbol z_k^-)\\
   	&=\frac{1}{2}\sum_{m,l=1,m\neq l}^j\kappa_m^+\kappa_l^+G(\boldsymbol z_m^+,\boldsymbol z_l^+)+\frac{1}{2}\sum_{l,n=1,l\neq n}^k\kappa_l^-\kappa_n^-G(\boldsymbol z_l^-,\boldsymbol z_n^-)\\
   	&\quad+\frac{1}{2}\sum_{m=1}^k(\kappa_m^+)^2H(\boldsymbol z_m^+,\boldsymbol z_m^+)+\frac{1}{2}\sum_{n=1}^j(\kappa_n^-)^2H(\boldsymbol z_n^-,\boldsymbol z_n^-)\\
   	&\quad-\sum_{m=1}^j\sum_{n=1}^k \kappa_m^+\kappa_n^- G(\boldsymbol z_m^+,\boldsymbol z_n^-)-\gamma\sum_{m=1}^j\kappa_m^+\cos \theta_m^++\gamma\sum_{n=1}^k\kappa_n^-\cos \theta_n^-.
   	\end{split}
   \end{equation}
   Now we are going to construct a series of vortex solutions to regularize $\omega^*(\boldsymbol z)$. According to Lemma \ref{lem1-1}, if we let $\psi_\ep=\psi-\gamma\cos\theta$ in \eqref{1-4}, then $\psi_\varepsilon$ should satisfy the equation
   \begin{equation}\label{1-8}
   	(-\Delta_{\mathbb S^2})\psi_\varepsilon=\frac{1}{\varepsilon^2}\sum_{m=1}^j\boldsymbol1_{B_\delta(\boldsymbol z_m^+)}\boldsymbol1_{\{\psi_\varepsilon+\gamma\cos\theta>\mu_{m,\varepsilon}^+\}}-\frac{1}{\varepsilon^2}\sum_{n=1}^k\boldsymbol1_{B_\delta(\boldsymbol z_n^-)}\boldsymbol1_{\{-\psi_\varepsilon-\gamma\cos\theta>\mu_{n,\varepsilon}^-\}},
   \end{equation}
   where $\delta>0$ is a small constant, $\mu_{m,\varepsilon}^+,\mu_{n,\varepsilon}^-$ are flux constants to be prescribed, and
   $$\{\psi_\varepsilon+\gamma\cos\theta>\mu_{m,\varepsilon}^+\}=\Omega_{m,\varepsilon}^+, \quad\quad \{-\psi_\varepsilon-\gamma\cos\theta>\mu_{n,\varepsilon}^-\}=\Omega_{n,\varepsilon}^-$$ 
   are the positive and negative level sets. The following theorem gives the regularization for singular vortex $\omega^*(\boldsymbol z)=\sum\limits_{m=1}^{j}\kappa_m^+\boldsymbol \delta_{\boldsymbol z_m^+}-\sum\limits_{n=1}^{k}\kappa_n^-\boldsymbol \delta_{\boldsymbol z_n^-}+2\gamma\cos\theta$.
   
    \begin{theorem}\label{thm2}
    	For any nondegenerate critical point $(\boldsymbol z_1^+,\cdots,\boldsymbol z_j^+, \boldsymbol z_1^-,\cdots, \boldsymbol z_k^-)$  of Kirchhoff--Routh function $\mathcal K_{k+j}$ defined by \eqref{1-7}, there exists an $\varepsilon_0>0$ small, such that for each $\varepsilon\in (0,\varepsilon_0]$, \eqref{1-8} has a solution $\psi_\varepsilon$. Let $\omega_\varepsilon=(-\Delta_{\mathbb S^2})\psi_\varepsilon+2\gamma\cos\theta$ be the vorticity function. We have the following asymptotic estimates:
    	\begin{itemize}
    		\item[(i)] One has
    		$$\omega_\varepsilon\rightharpoonup \sum\limits_{m=1}^{j}\kappa_m^+\boldsymbol \delta_{\boldsymbol z_m^+}-\sum\limits_{n=1}^{k}\kappa_n^-\boldsymbol \delta_{\boldsymbol z_n^-}+2\gamma\cos\theta \quad as \ \varepsilon\to 0^+,$$
    		where the convergence is in the sense of measures, and $\kappa_l^\pm$ is the circulation for each point vortex.
    		\item[(ii)] The boundaries of level sets $\Omega_{m,\varepsilon}^+$, $\Omega_{n,\varepsilon}^-$, are $C^1$ close curves, which are parameterized as
    		$$\partial \Omega_{m,\varepsilon}^+=\left\{\boldsymbol z_{m,\varepsilon}^++\big[\sqrt{\kappa_m^+/\pi}\varepsilon+o(\varepsilon)\big](\cos\xi,\sin^{-1}\theta_{m,\ep}^+\sin\xi) \mid \xi\in[0, 2\pi) \right\},$$
    		and
    		$$\partial \Omega_{n,\varepsilon}^-=\left\{\boldsymbol z_{n,\varepsilon}^-+\big[\sqrt{\kappa_n^-/\pi}\varepsilon+o(\varepsilon)\big](\cos\xi,\sin^{-1}\theta_{n,\ep}^-\sin\xi) \mid \xi\in[0, 2\pi) \right\}$$
    		with
    		$$\left(\boldsymbol z_{1,\varepsilon}^+,\cdots,\boldsymbol z_{j,\varepsilon}^+, \boldsymbol z_{1,\varepsilon}^-,\cdots, \boldsymbol z_{k,\varepsilon}^-\right)\to \left(\boldsymbol z_1^+,\cdots,\boldsymbol z_j^+, \boldsymbol z_1^-,\cdots, \boldsymbol z_k^-\right)\quad as \ \varepsilon\to 0^+.$$
    	\end{itemize}
    \end{theorem}
    
   In Theorem \ref{thm1}, we fix the location $\boldsymbol z_i^\pm$ of each vortex patch and adjust the traveling angular speed 
   $W_\varepsilon$. While in Theorem \ref{thm2}, we fix the rotating speed of the sphere $\gamma$ and find the proper condition 
   for the location vector $(\boldsymbol z_{1,\varepsilon}^+,\cdots,\boldsymbol z_{j,\varepsilon}^+, \boldsymbol z_{1,\varepsilon}^-,\cdots, \boldsymbol z_{k,\varepsilon}^-)$.
   The aim of these two procedures is similar, which is to eliminate the degenerate direction of a linearized operator and make $\psi_\varepsilon$ a solution
   to the primal problem \eqref{1-6} or \eqref{1-8}.
    
    \subsection{Organization of the paper}
    
    This paper is organized as follows. In Section \ref{sec2}, we obtain a series of patch solutions to regularize the K\'arm\'an point vortex street on $\mathbb S^2$. 
    To this purpose, we first construct suitable approximate solutions via the stream function of the well-known Rankine vortex and transform the problem into 
    a semilinear equation for the perturbation function. Then, we solve the projective problem module one-dimensional kernel of a linearized operator. 
    The proof is finished by finding the condition for traveling speed $W_\varepsilon$ and making the projection operator equal identity.
    This method is known as the Lyapunov--Schmidt reduction procedure. The regularization for a general steady vortex-wave system on $\mathbb S^2$ is 
    discussed in Section \ref{sec3}, where a $(2j+2k)$-dimensional problem is solved to determine the vortex location 
    $(\boldsymbol z_{1,\varepsilon}^+,\cdots,\boldsymbol z_{j,\varepsilon}^+, \boldsymbol z_{1,\varepsilon}^-,\cdots, \boldsymbol z_{k,\varepsilon}^-)$.
    
	\section{Regularization for K\'arm\'an point vortex street on $\mathbb S^2$}\label{sec2}	
	Since the regularized vorticity $\omega_\varepsilon$ concentrates to 
	$$\kappa\sum\limits_{i=1}^{k}\boldsymbol \delta_{\boldsymbol z_i^+}-\kappa\sum\limits_{i=1}^{k}\boldsymbol \delta_{\boldsymbol z_i^-},$$
	the asymptotic behavior of the stream function $\psi_\varepsilon$ far from the vorticity set $\Omega_\varepsilon$ is given by 
	$$\kappa\sum\limits_{i=1}^{k}G(\boldsymbol z,\boldsymbol z_i^+)-\kappa\sum\limits_{i=1}^{k}G(\boldsymbol z,\boldsymbol z_i^-).$$ 
	Our first task is to modify the Green kernel and to find an approximate solution, which relies on the decomposition \eqref{1-1} of $G(\boldsymbol z,\boldsymbol z')$ 
	at $\boldsymbol z_i^\pm$.  In the following, we are going to approximate
	 $\kappa\Gamma(\boldsymbol z, \boldsymbol z_i^\pm)$ and $\kappa H(\boldsymbol z, \boldsymbol z_i^\pm)$ in \eqref{1-1} separately, 
	 and make up a vortex patch solution on $\mathbb S^2$.
	
	\subsection{The approximate stream function}
	
	The equation for the stream function of the well-known Rankine vortex is 
	\begin{equation*}
		\begin{cases}
			\Delta w=\boldsymbol{1}_{B_1(\boldsymbol 0)},  \ \ \ & \text{in} \ \mathbb R^2,\\
			w=\frac{1}{2}\ln\frac{1}{|\boldsymbol y|}, &\text{in} \ \mathbb R^2\setminus B_{1}(\boldsymbol 0),
		\end{cases}
	\end{equation*}
	whose solution is given by
	\begin{equation*}
		w(\boldsymbol y)=\left\{
		\begin{array}{lll}
			\frac{1}{4}(1-|\boldsymbol y|^2), \ \ \ \ \ &\mathrm{if} \ |\boldsymbol y|\le 1,\\
			\\
			\frac{1}{2}\ln\frac{1}{|\boldsymbol y|}, &\mathrm{if} \ |\boldsymbol y|\ge 1.
		\end{array}
		\right.
	\end{equation*}
	The approximate solution for $\Gamma(\boldsymbol z, \boldsymbol z_i^\pm)$ looks like a scaled version of $w$. 
	Define the tangent mapping $A: (\theta,\varphi)\to (x_1,x_2)$ from $\mathbb S^2$ to $T_{\boldsymbol z}{\mathbb S^2}$ with the matrix
	\begin{equation*}
		\mathrm{Mat}(A)=\left(
		\begin{array}{ccc}
			1 & 0             \\
			0 & \sin\theta
		\end{array}
		\right).
	\end{equation*}
	Recall that the locations $\boldsymbol z_i^\pm$ satisfy $\mathbf{(type\, 1)}$ or $\mathbf{(type\, 2)}$ with $0<\theta_0<\pi/2$. We introduce the function 
	\begin{equation*}
		V^\pm_{i,\varepsilon}(\boldsymbol z)=\left\{
		\begin{array}{lll}
			\frac{\kappa}{2\pi}\ln\frac{1}{\varepsilon}+\frac{1}{4\varepsilon^2}(s_\varepsilon^2-|A(\boldsymbol z-\boldsymbol z^\pm_i)|^2), \ \ \ &\mathrm{if} \ |A(\boldsymbol z-\boldsymbol z^\pm_i)|\le s_\varepsilon,\\
			\frac{\kappa}{2\pi}\frac{|\ln\varepsilon|}{|\ln s_\varepsilon|}\ln|A(\boldsymbol z-\boldsymbol z^\pm_i)|,&\mathrm{if} \ |A(\boldsymbol z-\boldsymbol z^\pm_i)|\ge s_\varepsilon
		\end{array}
		\right.
	\end{equation*}
	as the approximate function for $\kappa\Gamma(\boldsymbol z, \boldsymbol z_i^\pm)$, where $|\cdot|$ is the distance in the tangent space 
	$T_{\boldsymbol z}{\mathbb S^2}$. To make $V_{i,\varepsilon}^\pm$ a $C^1$ smooth function, direct calculation gives the following relationship.
	\begin{equation}\label{2-1}
		\beta_\varepsilon:=\frac{\kappa}{2\pi}\frac{|\ln\varepsilon|}{s_\varepsilon|\ln s_\varepsilon|}=\frac{s_\varepsilon}{2\varepsilon^2},
	\end{equation}
	where $\beta_\varepsilon$ is the value of $-\partial_\theta V_{i,\varepsilon}^\pm$ at $(s_\varepsilon+\theta_0,\varphi_i^\pm)$. By the formulation of $V_{i,\varepsilon}^\pm$, it holds the following integration equation.
	\begin{equation*}
	   V^\pm_{i,\varepsilon}(\boldsymbol z)=\frac{1}{\varepsilon^2}\int_{\{V^\pm_{i,\varepsilon}(\boldsymbol z)>\frac{\kappa}{2\pi}\ln\frac{1}{\varepsilon}\}} \Gamma(\theta,\varphi,\theta',\varphi')d\boldsymbol \sigma(\boldsymbol z').
	\end{equation*}
	To approximate the remaining regular part $\kappa H(\boldsymbol z, \boldsymbol z_i^\pm)$, we introduce
	\begin{equation*}
		R^\pm_{i,\varepsilon}(\boldsymbol z)=\frac{1}{\varepsilon^2}\int_{\{V^\pm_{i,\varepsilon}(\boldsymbol z)>\frac{\kappa}{2\pi}\ln\frac{1}{\varepsilon}\}} H(\theta,\varphi,\theta',\varphi')d\boldsymbol \sigma(\boldsymbol z').
	\end{equation*} 
    Since $H(\boldsymbol z,\boldsymbol z')$ is $C^1$ smooth, we can also use the symmetry in the integration to derive that $R^\pm_{i,\varepsilon}(\boldsymbol z)$ is an $O(\varepsilon^2)$-perturbation of $\kappa H(\boldsymbol z,\boldsymbol z^\pm_i)$. 
    
    By the definitions of $V^\pm_{i,\varepsilon}(\boldsymbol z)$ and $R^\pm_{i,\varepsilon}(\boldsymbol z)$, we have
	\begin{equation*}
		(-\Delta_{\mathbb S^2})\left(V^\pm_{i,\varepsilon}+R^\pm_{i,\varepsilon}\right)=\frac{1}{\varepsilon^2}\boldsymbol1_{\{V^\pm_{i,\varepsilon}(\boldsymbol z)>\frac{\kappa}{2\pi}\ln\frac{1}{\varepsilon}\}}.
	\end{equation*}
	Thus we split $\psi_\varepsilon(\boldsymbol z)$ as
	\begin{align*}
		\psi_\varepsilon(\boldsymbol z)&=\sum\limits_{i=1}^kV^+_{i,\varepsilon}+\sum\limits_{i=1}^kR^+_{i,\varepsilon}-\sum\limits_{i=1}^kV^-_{i,\varepsilon}-\sum\limits_{i=1}^kR^-_{i,\varepsilon}+\phi_\varepsilon\\
		&:= \Psi_\varepsilon+\phi_\varepsilon,
	\end{align*}
	in which $\Psi_\varepsilon(\boldsymbol z)$ is the approximate solution, and $\phi_\varepsilon(\boldsymbol z)$ is a small error term. Note that it holds 
	\begin{equation}\label{2-2}
	\left\{\boldsymbol z\in \mathbb S^2 \; \left\vert \; V^\pm_{i,\varepsilon}(\boldsymbol z)>\frac{\kappa}{2\pi}\ln\frac{1}{\varepsilon}\right.\right\}=\{\boldsymbol z\in \mathbb S^2\mid |A(\boldsymbol z-\boldsymbol z^\pm_i)|< s_\varepsilon\}.
	\end{equation}
	Denote $\mathcal C_{s_\varepsilon}(\boldsymbol z_i^\pm)=\{\boldsymbol z\in \mathbb S^2\mid |A(\boldsymbol z-\boldsymbol z^\pm_i)|= s_\varepsilon\}$, 
	and $B_\delta(\boldsymbol z_i^\pm)\subset \mathbb S^2$ as the spherical cap region with radius $\delta$ to localize each vortex. 
	Then equation \eqref{1-6} is then transformed to
	\begin{align*}
			0&=-\varepsilon^2\Delta_{\mathbb S^2}\left(\sum\limits_{i=1}^kV^+_{i,\varepsilon}+\sum\limits_{i=1}^kR^+_{i,\varepsilon}-\sum\limits_{i=1}^kV^-_{i,\varepsilon}-\sum\limits_{i=1}^kR^-_{i,\varepsilon}+\phi_\varepsilon\right)\\
			&\quad-\boldsymbol1_{\{\psi_\varepsilon+W_\varepsilon\cos\theta>\mu_\varepsilon\}}+\boldsymbol1_{\{-\psi_\varepsilon-W_\varepsilon\cos\theta>\mu_\varepsilon\}}\\
			&=\sum\limits_{i=1}^k\left(-\varepsilon^2\Delta_{\mathbb S^2}\left(V^+_{i,\varepsilon}+R^+_{i,\varepsilon}\right)-\boldsymbol1_{\{V^+_{i,\varepsilon}>\frac{\kappa}{2\pi}\ln\frac{1}{\varepsilon}\}}\right)\\
			&\quad -\sum\limits_{i=1}^k\left(-\varepsilon^2\Delta_{\mathbb S^2}\left(V^-_{i,\varepsilon}+R^-_{i,\varepsilon}\right)-\boldsymbol1_{\{V^-_{i,\varepsilon}>\frac{\kappa}{2\pi}\ln\frac{1}{\varepsilon}\}}\right)\\
			& \ \ \ +\varepsilon^2\left(-\Delta_{\mathbb S^2}\phi_\varepsilon-\sum\limits_{i=1}^k\frac{2}{s_\varepsilon}\phi_\varepsilon\boldsymbol\delta_{\mathcal C_{s_\varepsilon}(\boldsymbol z_i^+)}-\sum\limits_{i=1}^k\frac{2}{s_\varepsilon}\phi_\varepsilon\boldsymbol\delta_{\mathcal C_{s_\varepsilon}(\boldsymbol z_i^-)}\right)\\
			& \ \ \ -\sum\limits_{i=1}^k\bigg(\boldsymbol1_{ B_\delta(\boldsymbol z_i^+)}\boldsymbol1_{\{\psi_\varepsilon+W_\varepsilon\cos\theta>\mu_\varepsilon\}}-\boldsymbol1_{\{V_{i,\varepsilon}^+>\frac{\kappa}{2\pi}\ln\frac{1}{\varepsilon}\}}-\frac{2}{s_\varepsilon}\phi_\varepsilon\boldsymbol\delta_{\mathcal C_{s_\varepsilon}(\boldsymbol z_i^+)}\bigg)\\
			& \ \ \ +\sum\limits_{i=1}^k\bigg(\boldsymbol1_{ B_\delta(\boldsymbol z_i^-)}\boldsymbol1_{\{-\psi_\varepsilon-W_\varepsilon\cos\theta>\mu_\varepsilon\}}-\boldsymbol1_{\{V_{i,\varepsilon}^->\frac{\kappa}{2\pi}\ln\frac{1}{\varepsilon}\}}+\frac{2}{s_\varepsilon}\phi_\varepsilon\boldsymbol\delta_{\mathcal C_{s_\varepsilon}(\boldsymbol z_i^-)}\bigg)\\
			&=\varepsilon^2\mathbb L_\varepsilon\phi_\varepsilon-\varepsilon^2N_\varepsilon(\phi_\varepsilon),
	\end{align*}
	where
	\begin{equation*}
		\mathbb L_\varepsilon\phi_\varepsilon:=(-\Delta_{\mathbb S^2})\phi_\varepsilon-\sum\limits_{i=1}^k\frac{2}{s_\varepsilon}\phi_\varepsilon\boldsymbol\delta_{\mathcal C_{s_\varepsilon}(\boldsymbol z_i^+)}-\sum\limits_{i=1}^k\frac{2}{s_\varepsilon}\phi_\varepsilon\boldsymbol\delta_{\mathcal C_{s_\varepsilon}(\boldsymbol z_i^-)}
	\end{equation*}
	is the linear term, and
	\begin{align*}
		N_\varepsilon(\phi_\varepsilon)=&\frac{1}{\varepsilon^2}\sum\limits_{i=1}^k\bigg(\boldsymbol1_{ B_\delta(\boldsymbol z_i^+)}\boldsymbol1_{\{\psi_\varepsilon+W_\varepsilon\cos\theta>\mu_\varepsilon\}}-\boldsymbol1_{\{V_{i,\varepsilon}^+>\frac{\kappa}{2\pi}\ln\frac{1}{\varepsilon}\}}-\frac{2}{s_\varepsilon}\phi_\varepsilon\boldsymbol\delta_{\mathcal C_{s_\varepsilon}(\boldsymbol z_i^+)}\bigg)\\
		&-\frac{1}{\varepsilon^2}\sum\limits_{i=1}^k\bigg(\boldsymbol1_{B_\delta(\boldsymbol z_i^-)}\boldsymbol1_{\{-\psi_\varepsilon-W_\varepsilon\cos\theta>\mu_\varepsilon\}}-\boldsymbol1_{\{V_{i,\varepsilon}^->\frac{\kappa}{2\pi}\ln\frac{1}{\varepsilon}\}}+\frac{2}{s_\varepsilon}\phi_\varepsilon\boldsymbol\delta_{\mathcal C_{s_\varepsilon}(\boldsymbol z_i^-)}\bigg)\\
		:=&N_\varepsilon^+(\phi_\varepsilon)-N_\varepsilon^-(\phi_\varepsilon)
	\end{align*}
	is the nonlinear perturbation. To make $N_\varepsilon(\phi_\varepsilon)$ sufficiently small, the flux constant $\mu_\varepsilon$ is determined by
	\begin{equation}\label{2-3}
		\begin{split}
		-\frac{\kappa}{2\pi}\ln\frac{1}{\varepsilon}=&\kappa\sum\limits_{i=2}^kG(\boldsymbol z_1^+,\boldsymbol z_i^+)-\kappa\sum\limits_{i=1}^kG(\boldsymbol z_1^+,\boldsymbol z_i^-)\\
		&+\kappa H(\boldsymbol z_1^+,\boldsymbol z_1^+)-W_\varepsilon\cos\theta_0-\mu_\varepsilon.
		\end{split}
	\end{equation}
	Having done all these preparations, we are now to solve the semilinear problem
	\begin{equation}\label{2-4}
		\mathbb L_\varepsilon\phi_\varepsilon=N_\varepsilon(\phi_\varepsilon),
	\end{equation}
	which is discussed in the next part.

\subsection{The linear problem}

If $\mathbb L_\varepsilon$ is invertable, we can transform \eqref{2-4} into a fixed point problem $\phi_\varepsilon=\mathbb L_\varepsilon^{-1}N_\varepsilon(\phi_\varepsilon)$, which is easy to handle. However, $\mathbb L_\varepsilon$ is not invertible in our problem. 
Thus we must find the approximate kernel, and deal with the projective problem.

To begin with, let us consider the linearized operator for $-\Delta v-\boldsymbol 1_{\{v>0\}}$ at the well-known Rankine vortex as
\begin{equation}\label{2-5}
	\mathbb L_0\phi:=-\Delta\phi-2\phi\boldsymbol\delta_{|\boldsymbol{y}|=1}.
\end{equation}
In \cite{Cao3}, the authors dealt with a eigenvalue problem for Laplace--Beltrami operator on $\mathbb S^1$ to show that 
\begin{theorem}\label{thm3}
	Let $w\in L^\infty(\mathbb R^2)\cap C(\mathbb R^2)$ be a solution to $\mathbb L _0w=0$ with the operator $\mathbb L_0$ defined in \eqref{2-5}. Then
	\begin{equation*}
		w\in\mathrm{span} \left\{\frac{\partial w}{\partial y_1},\frac{\partial w}{\partial y_2}\right\}.
	\end{equation*}
\end{theorem}

By defining the mapping $A_0: (\theta,\varphi)\to (x_1,x_2)$ from $\mathbb S^2$ to $T_{\boldsymbol z_i^\pm}\mathbb S^2$ with the matrix 
\begin{equation*}
	\mathrm{Mat}(A_0)=\left(
	\begin{array}{ccc}
		1 & 0             \\
		0 & \sin\theta_0
	\end{array}
	\right),
\end{equation*}
we introduce the rescaled Rankine vortex $U^\pm_{i,\varepsilon}(\boldsymbol z)$ as 
\begin{equation*}
	U^\pm_{i,\varepsilon}(\boldsymbol z)=\left\{
	\begin{array}{lll}
		\frac{\kappa}{2\pi}\ln\frac{1}{\varepsilon}+\frac{1}{4\varepsilon^2}(s_\varepsilon^2-|A_0(\boldsymbol z-\boldsymbol z^\pm_i)|^2), \ \ \ &\mathrm{if} \ |A_0(\boldsymbol z-\boldsymbol z^\pm_i)|\le s_\varepsilon,\\
		\frac{\kappa}{2\pi}\frac{|\ln\varepsilon|}{|\ln s_\varepsilon|}\ln|A_0(\boldsymbol z-\boldsymbol z^\pm_i)|,&\mathrm{if} \ |A_0(\boldsymbol z-\boldsymbol z^\pm_i)|\ge s_\varepsilon,
	\end{array}
	\right.
\end{equation*}
which are functions near $V^\pm_{i,\varepsilon}(\boldsymbol z)$. Since $\mathbb L_\varepsilon$ is almost the scaled version of $\mathbb L_0$, a suitable choice for the element in the approximate kernel is
\begin{equation*}
	Z_\varepsilon ({\boldsymbol z})=\sum\limits_{i=1}^k\chi_i^+ ({\boldsymbol z})\frac{\partial U^+_{i,\varepsilon}}{\partial \theta}-\sum\limits_{i=1}^k\chi_i^- ({\boldsymbol z})\frac{\partial U^-_{i,\varepsilon}}{\partial \theta},
\end{equation*}
where
\begin{equation*}
	\frac{\partial U^+_{i,\varepsilon}}{\partial \theta}=\left\{
	\begin{array}{lll}
		-\frac{1}{2\varepsilon^2}(\theta-\theta_0), \ \ \ & \mathrm{if} \ |A_0(\boldsymbol z-\boldsymbol z^+_i)|\le s_\varepsilon,\\
		-\frac{\kappa|\ln \varepsilon|}{2\pi|\ln s_\varepsilon|}\frac{\theta-\theta_0}{|A_0(\boldsymbol z-\boldsymbol z^+_i)|^2}, & \mathrm{if} \ |A_0(\boldsymbol z-\boldsymbol z^+_i)|\ge s_\varepsilon,
	\end{array}
	\right.
\end{equation*}
\begin{equation*}
	\frac{\partial U^-_{i,\varepsilon}}{\partial \theta}=\left\{
	\begin{array}{lll}
		-\frac{1}{2\varepsilon^2}(\theta-\theta_0), \ \ \ & \mathrm{if} \ |A_0(\boldsymbol z-\boldsymbol z^-_i)|\le s_\varepsilon,\\
		-\frac{\kappa|\ln \varepsilon|}{2\pi|\ln s_\varepsilon|}\frac{\theta-(\pi-\theta_0)}{|A_0(\boldsymbol z-\boldsymbol z^-_i)|^2}, & \mathrm{if} \ |A_0(\boldsymbol z-\boldsymbol z^-_i)|\ge s_\varepsilon,
	\end{array}
	\right.
\end{equation*}
and 
\begin{equation*}
	\chi_i^\pm(\boldsymbol z)=\left\{
	\begin{array}{lll}
		1, \ \ \  & \mathrm{if} \ |\boldsymbol z-\boldsymbol z_i^\pm|_{\mathbb S^2}< \varepsilon|\ln\varepsilon|,\\
		0, & \mathrm{if} \ |\boldsymbol z-\boldsymbol z_i^\pm|_{\mathbb S^2}\ge 2\varepsilon|\ln\varepsilon|
	\end{array}
	\right.
\end{equation*}
are smooth truncation functions radially symmetric with respect to $\boldsymbol z_i^\pm$ satisfying
\begin{equation*}
	|\nabla \chi_i^\pm(\boldsymbol z)|\le \frac{2}{\varepsilon|\ln\varepsilon|} \quad \mathrm{and} \quad |\nabla^2 \chi_i^\pm(\boldsymbol z)|\le \frac{2}{\varepsilon^2|\ln\varepsilon|^2}
\end{equation*}
 with $|\cdot|_{\mathbb S^2}$ the distance on the sphere $\mathbb S^2$. We claim that the approximate kernel of $\mathbb L_\varepsilon$ is
 one-dimensional, since the direction $\partial_\varphi$ is eliminated by the $k$-fold symmetry of the vertex streets.

Also thanks to the $k$-fold symmetry of the problem, we can restrict our discussion to the domain 
\begin{equation*}
	\Pi:=(0,\pi)\times\left(0,2\pi/k\right),
\end{equation*}
which is regarded as a typical period of K\'arm\'an vortex streets. In these settings, the projective problem for \eqref{2-4} is rewritten as
\begin{equation}\label{2-6}
	   \begin{cases}
		\mathbb L_\varepsilon\phi=\mathbf h(\boldsymbol z)+\Lambda(-\Delta_{\mathbb S^2})Z_\varepsilon, \ \ &\text{in} \ \Pi,\\
		\int_{\Pi}  \phi(\boldsymbol z)(-\Delta_{\mathbb S^2})Z_\varepsilon(\boldsymbol z) d\boldsymbol \sigma=0,
	   \end{cases}
\end{equation}
where the nonlinear term $\mathbf h(\boldsymbol z)$ satisfies 
$$\mathrm{supp}\, \mathbf h(\boldsymbol z)\subset B_{L\varepsilon}(\boldsymbol z_1^+)\cup B_{L\varepsilon}(\boldsymbol z_1^-)$$
with $L$ a large positive constant, and $\Lambda$ is the coefficient such that 
\begin{equation*}
	\int_{\Pi}Z_\varepsilon\big[\mathbb L_\varepsilon\phi-\mathbf h-\Lambda (-\Delta_{\mathbb S^2})Z_\varepsilon\big]d\boldsymbol \sigma=0.
\end{equation*}

By denoting the $\|\cdot\|_*$ norm for the function $\phi$ on the sphere as  
$$\|\phi\|_*=\sup_\Pi |\phi|,$$
we first give a coercive estimate for $\mathbb L_\varepsilon$ in the next lemma, from which we can verify the choice of $Z_\varepsilon$ is reasonable. 
In the subsequent proof, the large positive constant $L$ is always used to bound the radius of each vortex set $\Omega_{i,\varepsilon}^\pm$ with $L\ep$. 
\begin{lemma}\label{lem2-2}
	Assume that $\mathbf h$ satisfies $\mathrm{supp}\, \mathbf h\subset B_{L\varepsilon}(\boldsymbol z_1^+)\cup B_{L\varepsilon}(\boldsymbol z_1^-),$ and $$\varepsilon^{1-\frac{2}{p}}\|\mathbf h\|_{W^{-1,p}(B_{L\varepsilon}(\boldsymbol z_1^+)\cup B_{L\varepsilon}(\boldsymbol z_1^-))}<\infty$$
	for $p\in (2,+\infty]$, then there exists a small $\varepsilon_0>0$ and a positive constant $c_0$ such that for any $\varepsilon\in(0,\varepsilon_0]$ and solution pair $(\phi,\Lambda)$ to \eqref{2-6}, it holds
	\begin{equation}\label{2-7}
		\|\phi\|_*+\varepsilon^{1-\frac{2}{p}}\|\nabla\phi\|_{L^p(B_{L\varepsilon}(\boldsymbol z_1^+)\cup B_{L\varepsilon}(\boldsymbol z_1^-))}+|\Lambda|\varepsilon^{-1}\le c_0\varepsilon^{1-\frac{2}{p}}\|\mathbf h\|_{W^{-1,p}(B_{L\varepsilon}(\boldsymbol z_1^+)\cup B_{L\varepsilon}(\boldsymbol z_1^-))}.
	\end{equation}
\end{lemma}
\begin{proof}
	To derive the estimate for coefficient $\Lambda$, we integrate by parts to obtain
	\begin{equation}\label{2-8}
		\Lambda\int_{\Pi} \left[(\partial_\theta Z_\varepsilon)^2+\left(\frac{\partial_\varphi Z_\varepsilon}{\sin\theta}\right)^2\right]d\boldsymbol \sigma=\int_{\Pi}Z_\varepsilon\mathbb L_\varepsilon\phi d\boldsymbol \sigma-\int_{\Pi} Z_\varepsilon\mathbf h d\boldsymbol \sigma.
	\end{equation}
Let 
	$$\Pi^+:=(0,\pi/2)\times\left(0,2\pi/k\right)$$ 
	be the upper half of $\Pi$. For the left hand side of \eqref{2-8}, it holds
	\begin{align*}
		\Lambda\int_{\Pi} \left[(\partial_\theta Z_\varepsilon)^2+\left(\frac{\partial_\varphi Z_\varepsilon}{\sin\theta}\right)^2\right]d\boldsymbol \sigma
		&=2\Lambda\int_{\Pi^+} \sin\theta \left[\left(\frac{\partial^2 U^+_{1,\varepsilon}}{\partial \theta^2}\right)^2+\left(\frac{\frac{\partial^2U_{1,\varepsilon}^+}{\partial_\varphi\partial_\theta}}{\sin\theta_0}\right)^2\right]d\theta d\varphi+\Lambda\frac{o_\varepsilon(1)}{\varepsilon^2}\\
		&=\Lambda\frac{C_z}{\varepsilon^2}(1+o_\varepsilon(1)),
	\end{align*}
	where $C_z=\int_{\mathbb R^2}(\nabla \partial_{y_1}w)^2 d\boldsymbol y$ is a positive constant.

	To estimate the first term in the right-hand side of \eqref{2-8}, it is sufficient to consider the integration in $\Pi^+$ owing to the symmetry of the vortex streets.
	\begin{align*}
		\int_{\Pi^+}Z_\varepsilon\mathbb L_\varepsilon\phi d\boldsymbol \sigma&=\int_{\Pi^+}\phi\mathbb L_\varepsilon Z_\varepsilon d\boldsymbol \sigma\\
		&=\int_{\Pi^+}\phi\mathbb (-\Delta_{\mathbb S^2})Z_\varepsilon d\boldsymbol \sigma-\frac{2}{s_\varepsilon}\int_{\mathcal C_{s_\varepsilon}(\boldsymbol z_1^+)}\phi\sin\theta \frac{\partial U^+_{1,\varepsilon}}{\partial \theta}.
	\end{align*}
	It holds
	\begin{align*}
		\int_{\Pi^+}&\phi\mathbb (-\Delta_{\mathbb S^2})Z_\varepsilon d\boldsymbol \sigma=\int_{\Pi^+}\phi\left[\partial_\theta(\sin\theta\partial_\theta Z_\varepsilon)+\frac{\partial_\varphi^2Z_\varepsilon}{\sin\theta}\right]d\theta d\varphi\\
		&=\int_{\Pi^+}\phi\cos\theta\partial_\theta\left(\chi_1^+\frac{\partial U^+_{1,\varepsilon}}{\partial \theta}\right)d\theta d\varphi +\int_{\Pi^+}\phi\sin\theta\left[(\partial_\theta^2\chi_1^+)\frac{\partial U^+_{1,\varepsilon}}{\partial \theta}+\frac{\partial_\varphi^2\chi_1^+}{\sin^2\theta}\frac{\partial U^+_{1,\varepsilon}}{\partial \theta}\right]d\theta d\varphi\\
		&\quad+2\int_{\Pi^+}\phi\sin\theta\left[\partial_\theta\chi_1^+\frac{\partial^2 U^+_{1,\varepsilon}}{\partial \theta^2}+\frac{\partial_\varphi\chi_1^+}{\sin^2\theta} \frac{\partial^2 U^+_{1,\varepsilon}}{\partial\varphi\partial \theta}\right]d\theta d\varphi\\
		&\quad+\int_{\Pi^+}\phi\chi_1^+\sin\theta\left[\frac{1}{\sin^2\theta}-\frac{1}{\sin^2\theta_0}\right] \frac{\partial^2 U^+_{1,\varepsilon}}{\partial\varphi^2\partial \theta}d\theta d\varphi+\int_{\Pi^+}\phi\chi_1^+\sin\theta\left[\frac{\partial^3 U^+_{1,\varepsilon}}{\partial \theta^3}+\frac{\frac{\partial^3 U^+_{1,\varepsilon}}{\partial \varphi^2\partial \theta}}{\sin^2\theta_0}\right]d\theta d\varphi\\
		&=I_1+I_2+I_3+I_4+I_5.
\end{align*}
	According to the property for truncation function $\chi_1^+$, we have $I_1\le C|\ln\varepsilon|\|\phi\|_*$, $I_2\le \frac{C\|\phi\|_*}{\varepsilon|\ln\varepsilon|}$, $I_3\le \frac{C\|\phi\|_*}{\varepsilon|\ln\varepsilon|}$, and $I_4\le C\|\phi\|_*$. For the last term $I_5$, by applying the nondegenerate property in Theorem \ref{thm3}, we have
	\begin{equation*}
		I_5=\frac{2}{s_\varepsilon}\int_{\mathcal C_{s_\varepsilon}(\boldsymbol z_1^+)}\phi\sin\theta \frac{\partial U^+_{1,\varepsilon}}{\partial \theta}
		+O_\varepsilon(1)\|\phi\|_*.
	\end{equation*} 
	Thus we have
	\begin{equation*}
		\int_{\Pi^+}Z_\varepsilon\mathbb L_\varepsilon\phi d\boldsymbol \sigma\le \frac{C}{\varepsilon|\ln\varepsilon|}\|\phi\|_*.
	\end{equation*}
	For the last term in \eqref{2-8}, we can apply the Poincar\'e inequality to derive 
	\begin{align*}
		\int_{\Pi} Z_\varepsilon\mathbf h d\boldsymbol \sigma&\le \|\mathbf h\|_{W^{-1,p}(B_{L\varepsilon}(\boldsymbol z_1^+)\cup B_{L\varepsilon}(\boldsymbol z_1^-))}\|\nabla Z_\varepsilon\|_{L^{p'}(B_{L\varepsilon}(\boldsymbol z_1^+)\cup B_{L\varepsilon}(\boldsymbol z_1^-))}\\
		&\le C\varepsilon^{\frac{2}{p'}-2} \|\mathbf h\|_{W^{-1,p}(B_{L\varepsilon}(\boldsymbol z_1^+)\cup B_{L\varepsilon}(\boldsymbol z_1^-))}.
	\end{align*}
	Concluding all the above estimates together, we have
	\begin{equation*}
		|\Lambda|\varepsilon^{-1}\le\frac{C}{|\ln\varepsilon|}\|\phi\|_*+\varepsilon^{\frac{2}{p'}-1}\|\mathbf h\|_{W^{-1,p}(B_{L\varepsilon}(\boldsymbol z_1^+)\cup B_{L\varepsilon}(\boldsymbol z_1^-))}.
	\end{equation*}
	Since
	\begin{equation}\label{2-9}
		\|(-\Delta_{\mathbb S^2})Z_\varepsilon\|_{W^{-1,p}(B_{L\varepsilon}(\boldsymbol z_1^+)\cup B_{L\varepsilon}(\boldsymbol z_1^-))}\le C\|\nabla Z_\varepsilon\|_{L^{p}(B_{L\varepsilon}(\boldsymbol z_1^+)\cup B_{L\varepsilon}(\boldsymbol z_1^-))}=C \varepsilon^{\frac{2}{p}-2},
	\end{equation} 
	it holds
	\begin{equation*}
			\|\Lambda(-\Delta_{\mathbb S^2})Z_\varepsilon\|_{W^{-1,p}(B_{L\varepsilon}(\boldsymbol z_1^+)\cup B_{L\varepsilon}(\boldsymbol z_1^-))}\le C \frac{\varepsilon^{\frac{2}{p}-1}}{|\ln\varepsilon|}\|\phi\|_*+C\|\mathbf h\|_{W^{-1,p}(B_{L\varepsilon}(\boldsymbol z_1^+)\cup B_{L\varepsilon}(\boldsymbol z_1^-))}.
	\end{equation*}
	
	To obtain the estimate for $\|\phi\|_*$ and $\varepsilon^{1-\frac{2}{p}}\|\nabla\phi\|_{L^p(B_{L\varepsilon}(\boldsymbol z_1^+)\cup B_{L\varepsilon}(\boldsymbol z_1^-))}$ in the next step, we argue by contradiction. Suppose not. Then there exists a sequence $\{\varepsilon_n\}$ tending to $0$, such that
	\begin{equation}\label{2-10}
		\|\phi_n\|_*+\varepsilon_n^{1-\frac{2}{p}}\|\nabla\phi_n\|_{L^p(B_{L\varepsilon_n}(\boldsymbol z_1^+)\cup B_{L\varepsilon_n}(\boldsymbol z_1^-))}=1,
	\end{equation} 
	and
	\begin{equation*}
		\varepsilon_n^{\frac{2}{p'}-2} \|\mathbf h\|_{W^{-1,p}(B_{L\varepsilon_n}(\boldsymbol z_1^+)\cup B_{L\varepsilon_n}(\boldsymbol z_1^-))}\le\frac{1}{n}.
	\end{equation*}
	The solution $\phi_n$ satisfies the equation
	\begin{equation*}
		(-\Delta_{\mathbb S^2})\phi_n=\frac{2}{s_\varepsilon}\phi_n\boldsymbol \delta_{\mathcal C_{s_\varepsilon}(\boldsymbol z_1^+)}+\frac{2}{s_\varepsilon}\phi_n\boldsymbol \delta_{\mathcal C_{s_\varepsilon}(\boldsymbol z_1^-)}+\mathbf h+\Lambda(-\Delta_{\mathbb S^2})Z_\varepsilon.
	\end{equation*}
	By letting
	$$g(\boldsymbol z)=\frac{2}{s_\varepsilon}\phi_n\delta_{\mathcal C_{s_\varepsilon}(\boldsymbol z_1^-)}+\mathbf h+\Lambda(-\Delta_{\mathbb S^2})Z_\varepsilon$$
	and $\tilde v(\boldsymbol y)=v(s_{\varepsilon_n}y_1+\theta_0,s_{\varepsilon_n}\sin^{-1}(s_{\varepsilon_n}y_1+\theta_0)y_2+\varphi_1^+)$ for a general function $v$, we see that $\tilde\phi_n$ satisfies the scaled equation
	\begin{equation}\label{2-11}
		\int_{D_n}\left[\frac{\partial \tilde\phi_n}{\partial y_1}\frac{\partial \zeta}{\partial y_1}+\frac{\partial \tilde\phi_n}{\partial y_2}\frac{\partial \zeta}{\partial y_2}\right]d\boldsymbol y=2\int_{|\boldsymbol y|=1}\tilde\phi_n\zeta d\xi+\langle \tilde g_n, \zeta\rangle, \quad \forall \, \zeta\in C_0^\infty(D_n),
	\end{equation}
	where
	$$D_n:=\{\boldsymbol y\mid (s_{\varepsilon_n}y_1+\theta_0,s_{\varepsilon_n}\sin^{-1}(s_{\varepsilon_n}y_1+\theta_0)y_2+\varphi_1^+)\in \Pi^+\}.$$
	Since 
	\begin{equation*}
		\|\tilde g_n\|_{W^{-1,p}(B_L(\boldsymbol 0))}\le C\varepsilon_n^{1-\frac{2}{p}}\left( \frac{\varepsilon_n^{\frac{2}{p}-1}}{|\ln\varepsilon_n|}\|\phi_n\|_*+\|\mathbf h\|_{W^{-1,p}(B_{L\varepsilon_n}(\boldsymbol z_1^+)\cup B_{L\varepsilon_n}(\boldsymbol z_1^-))}\right)=o_n(1)
	\end{equation*}
	by estimate \eqref{2-9}, $\tilde\phi_n$ is bounded in $W^{1,p}_{\mathrm{loc}}(\mathbb R^2)$ by regularity theory of elliptic operator, and bounded in $C^{\alpha}_{\mathrm{loc}}(\mathbb R^2)$ for some $\alpha>0$ by Sobolev embedding. Hence we can assume $\tilde\phi_n$ converge uniformly to $\phi^*\in L^\infty(\mathbb R^2)\cap C(\mathbb R^2)$ in any fixed compact set in $\mathbb R^2$, which satisfies the limit equation
	\begin{equation*}
	   -\Delta\phi^*=2\phi^*\boldsymbol \delta_{|\boldsymbol y|=1}, \quad \mathrm{in} \ \mathbb R^2.
	\end{equation*}
	Owing to Theorem \ref{thm3} and the symmetry property of our problem, it holds
	\begin{equation*}
		\phi^*=C^*\frac{\partial w}{\partial y_1}.
	\end{equation*}
	However, from the projection condition in \eqref{2-6} we know that $\int_{\mathbb R^2}\nabla\phi^*\nabla \partial_{y_1} wd\boldsymbol y=0$. Thus $C^*=0$ and
	\begin{equation*}
		\|\phi_n\|_{L^\infty(B_{L\varepsilon_n}(\boldsymbol z_1^+)\cup B_{L\varepsilon_n}(\boldsymbol z_1^-))}=o_n(1).
	\end{equation*}
	Then, in view of the maximum principle, we have
	\begin{equation}
		\|\phi_n\|_*=o_n(1).
	\end{equation}
	
	On the other hand, the right-hand side of equation \eqref{2-11} can be controlled by 
	$$o_n(1)\|\zeta\|_{W^{1,1}(B_L(\boldsymbol 0))}+o_n(1)\|\zeta\|_{W^{1,p'}(B_L(\boldsymbol 0))}=o_n(1)\left(\int_{B_L(\boldsymbol 0)}|\nabla\zeta|^{p'}\right)^{\frac{1}{p'}}d\boldsymbol y.$$
    As a result, we deduce that
    \begin{equation*}
    	\varepsilon_n^{1-\frac{2}{p}}\|\nabla\phi_n\|_{L^p(B_{L\varepsilon_n}(\boldsymbol z_1^+)\cup B_{L\varepsilon_n}(\boldsymbol z_1^-))}\le C||\nabla\tilde\phi_n||_{L^p(B_L(\boldsymbol 0))}=o_n(1).
    \end{equation*}
	Hence we get a contradiction to \eqref{2-10}. Combining this fact with the estimate for $\Lambda$, we then deduce that \eqref{2-7} holds, and complete the proof of Lemma \ref{lem2-2}.
\end{proof}
	
	By the coercive estimate in Lemma \ref{lem2-2}, the projective problem \eqref{2-6} is solvable according to the following lemma.
	\begin{lemma}\label{lem2-3}
		Suppose that $\mathrm{supp}\, \mathbf h\subset B_{L\varepsilon}(\boldsymbol z_1^+)\cup B_{L\varepsilon}(\boldsymbol z_1^-)$ and 
		$$\varepsilon^{1-\frac{2}{p}}\|\mathbf h\|_{W^{-1,p}(B_{L\varepsilon}(\boldsymbol z_1^+)\cup B_{L\varepsilon}(\boldsymbol z_1^-))}<\infty$$
		for $p\in(2,+\infty]$. Then there exists a small $\varepsilon_0>0$ such that for any $\varepsilon\in(0,\varepsilon_0]$ and $\Lambda$ the projection coefficient, \eqref{2-6} has a unique solution $\phi_\varepsilon=\mathcal T_\varepsilon \, \mathbf h$, where $\mathcal T_\varepsilon$ is a linear operator. Moreover, there exists a constant $c_0>0$ independent of $\varepsilon$, such that
		\begin{equation}\label{2-13}
			\|\phi_\varepsilon\|_*+\varepsilon^{1-\frac{2}{p}}\|\nabla\phi_\varepsilon\|_{L^p(B_{L\varepsilon}(\boldsymbol z_1^+)\cup B_{L\varepsilon}(\boldsymbol z_1^-))}\le c_0\varepsilon^{1-\frac{2}{p}}\|\mathbf h\|_{W^{-1,p}(B_{L\varepsilon}(\boldsymbol z_1^+)\cup B_{L\varepsilon}(\boldsymbol z_1^-))}.
		\end{equation}
	\end{lemma}
	\begin{proof}
		Let $H(\Pi)$ be the Hilbert space endowed with the inner product
		\begin{equation*}
			[u,v]_{H(\Pi)}=\int_{\Pi}\left[\frac{\partial u}{\partial\theta}\frac{\partial v}{\partial \theta}+\frac{1}{\sin^2\theta}\frac{\partial u}{\partial\varphi}\frac{\partial v}{\partial \varphi}\right]d\boldsymbol\sigma.
		\end{equation*}
		We also introduce two spaces as follows. The first one is
		\begin{equation*}
			E_\varepsilon:=\left\{u\in H(\Pi)\,\,\, \big|\, \,\, \|u\|_*<\infty, \ k\mathrm{-fold \ symmetric}, \ [u,Z_\varepsilon]_{H(\Pi)}=0\right\}
		\end{equation*}
		with norm $||\cdot||_*$, and the second one is
		\begin{equation*}
			F_\varepsilon:=\left\{\mathbf h^* \in W^{-1,p}(B_{L\varepsilon}(\boldsymbol z_1^+)\cup B_{L\varepsilon}(\boldsymbol z_1^-)) \,\,\, \big| \,\,\, \ k\mathrm{-fold \ symmetric}, \ \int_{\Pi}Z_\varepsilon\mathbf h^*d\boldsymbol \sigma=0\right\}
		\end{equation*}
		with $p\in (2,+\infty]$. Then for $\phi_\varepsilon\in E_\varepsilon$, the problem \eqref{2-6} has an equivalent operator form.
		\begin{equation*}
			\begin{split}
				\phi_\varepsilon&=(-\Delta_{\mathbb S^2})_*^{-1}P_\varepsilon\left(\frac{2}{s_\varepsilon }\phi_\varepsilon\boldsymbol{\delta}_{\mathcal C_{s_\varepsilon}(\boldsymbol z^+_1)}+\frac{2}{s_\varepsilon }\phi_\varepsilon\boldsymbol{\delta}_{\mathcal C_{s_\varepsilon}(\boldsymbol z^-_1)}\right)+(-\Delta_{\mathbb S^2})_*^{-1} P_\varepsilon \mathbf h\\
				&=\mathscr K\phi_\varepsilon+(-\Delta_{\mathbb S^2})_*^{-1} P_\varepsilon \mathbf h,
			\end{split}
		\end{equation*}
		where
		\begin{equation*}
			(-\Delta_{\mathbb S^2})_*^{-1}u:=\sum\limits_{i=0}^{k-1}\int_{\Pi}G((\theta,\varphi+2i\pi/k),\boldsymbol z')u(\boldsymbol z')d\boldsymbol \sigma',
		\end{equation*}
		and $P_\varepsilon$ is the projection operator to $F_\varepsilon$. Since $Z_\varepsilon$ has a compact support owing 
		to the truncation $\chi_i^\pm(\boldsymbol z)$, by the definition of ${G_*}(\boldsymbol x,\boldsymbol x')$, 
		we see that $\mathscr K$ maps $E_\varepsilon$ to $E_\varepsilon$.
		
		Notice that $\mathscr K$ is a compact operator. In view of the Fredholm alternative, \eqref{2-6} has a unique solution if the homogeneous equation
		\begin{equation*}
			\phi_\varepsilon=\mathscr K\phi_\varepsilon
		\end{equation*}
		has only trivial solution in $E_\varepsilon$, which can be obtained from Lemma \ref{lem2-2}. Now we set
		$$\mathcal T_\varepsilon:=(\text{Id}-\mathscr K)^{-1}(-\Delta_{\mathbb S^2})_*^{-1} P_\varepsilon,$$
		and the estimate \eqref{2-13} holds by Lemma \ref{lem2-2}. The proof is thus complete.
	\end{proof}
	
	\subsection{The reduction}
	In order to apply the contraction mapping theorem, we give a delicate estimate for the nonlinear term $N_\varepsilon(\phi_\varepsilon)$, which relies on an expansion near the vorticity domain $\Omega_{1,\varepsilon}^\pm$. 
	
	By denoting $\Pi^+:=(0,\pi/2)\times\left(0,2\pi/k\right)$, and 
	$$\tilde v(\boldsymbol y)=v(s_\varepsilon y_1+\theta_0,s_\varepsilon\sin^{-1}(s_\varepsilon y_1+\theta_0)y_2+\varphi_1^+)$$ for a general function $v$, we have following lemma concerning the estimate for the level set $\{\boldsymbol z\mid \psi_\varepsilon(\boldsymbol z)+W_\varepsilon\sin\theta_0=\mu_\varepsilon\}$.
	\begin{lemma}\label{lem2-4}
		Suppose that $\tilde\phi$ is a function satisfying
		\begin{equation}\label{2-14}
			\|\nabla\tilde\phi\|_{L^\infty(B_L(\boldsymbol 0))}+\|\tilde\phi\|_{L^\infty(B_L(\boldsymbol 0))}=O(\varepsilon).
		\end{equation}
		with $p\in(2,+\infty]$. Then the set
		$$\tilde{\mathbf\Gamma}_{\varepsilon,\tilde\phi}:=\{\boldsymbol y \mid \tilde\Psi_\varepsilon(\boldsymbol y)+\tilde\phi+W_\ep\cos(s_\varepsilon y_1+\theta_0)=\mu_\varepsilon\}$$
		is a closed convex curve in $\mathbb{R}^2$, and
		\begin{equation}\label{2-15}
			\begin{split}
				\tilde{\mathbf\Gamma}_{\varepsilon,\tilde\phi}(\xi)&=(1+t_\varepsilon(\xi))(\cos\xi,\sin\xi)\\
				&=(1+t_{\varepsilon,\tilde\phi}(\xi)+t_{\varepsilon,\mathcal H}(\xi)+O(\varepsilon^2))(\cos\xi,\sin\xi)\\
				&=\left(1+\frac{1}{s_\varepsilon\beta_\varepsilon}\tilde\phi(\cos\xi,\sin\xi)\right)(\cos\xi,\sin\xi)+\frac{\mathcal H(\xi)}{s_\varepsilon\beta_\varepsilon}(\cos\theta,0)\\
				&\quad+O(\varepsilon^2), \quad \xi\in (0,2\pi]
			\end{split}
		\end{equation}
		with $\|t_\varepsilon\|_{C^1[0,2\pi)}=O(\ep)$, and 
		\begin{align*}
			\mathcal H(\xi)&=\bigg\langle\kappa \sum\limits_{i=2}^k\nabla G(\boldsymbol z_1^+,\boldsymbol z^+_i)+\kappa \nabla H(\boldsymbol z_1^+,\boldsymbol z^+_1)-\kappa \sum\limits_{i=1}^k\nabla G(\boldsymbol z_1^+,\boldsymbol z^-_i)\\
			&\quad-(W_\ep\sin\theta_0,0),\quad (s_\varepsilon\cos\xi,s_\varepsilon\sin^{-1}(s_\varepsilon y_1+\theta_0)\sin\xi)\bigg\rangle.
		\end{align*}
	    Moreover, it holds
		\begin{equation}\label{2-16}
			\left|\tilde{\mathbf\Gamma}_{\varepsilon,\tilde\phi_1}-\tilde{\mathbf\Gamma}_{\varepsilon,\tilde\phi_2}\right|=\left(\frac{1}{s_\varepsilon\beta_\varepsilon}+O(\varepsilon)\right)\|\tilde\phi_1-\tilde\phi_2\|_{L^\infty(B_L(\boldsymbol 0))}
		\end{equation}
		for two functions $\tilde\phi_1,\tilde\phi_2$ satisfying \eqref{2-14}.
	\end{lemma}
	\begin{proof}
		Let $(y_1,y_2)=(\cos\xi,\sin\xi)$, and the inner product $\mathcal H(\xi)$ is a regular term of order $O(\varepsilon)$. Then by the definition of $\mu_\varepsilon$ in \eqref{2-3}, we have
		\begin{equation*}
			\tilde\psi_\varepsilon(\boldsymbol y)-W_\ep\cos(s_\varepsilon y_1+\theta_0)-\mu_\varepsilon=\tilde\phi+\mathcal H+O(\varepsilon^2).
		\end{equation*} 
		In view of \eqref{2-1} and \eqref{2-2}, it holds
		\begin{align*}
			&\nabla\left[\tilde\psi_\varepsilon(\boldsymbol y)-W_\ep\cos(s_\varepsilon y_1+\theta_0)-\mu_\varepsilon\right]\bigg|_{|\boldsymbol y|=1}\\
			=&\nabla\left[\tilde V^+_{1,\varepsilon}(\boldsymbol y)-\frac{\kappa}{2\pi}\ln\frac{1}{\varepsilon}\right]\bigg|_{|\boldsymbol y|=1}+O(\varepsilon)=s_\varepsilon\beta_\varepsilon+O(\varepsilon).
		\end{align*}
		Using the implicit function theorem, we have 
		$$\|t_{\varepsilon,\tilde\phi}\|_{C^1[0,2\pi)}=\frac{\| \tilde\phi\|_{C^1(B_L(\boldsymbol 0))}}{s_\varepsilon\beta_\varepsilon+O(\varepsilon)}=O(\varepsilon),$$ 
		and 
		$$\|t_{\varepsilon,\mathcal H}\|_{C^1[0,2\pi)}=\frac{\|\mathcal H\|_{C^1(B_L(\boldsymbol 0))}}{s_\varepsilon\beta_\varepsilon+O(\varepsilon)}=O(\varepsilon).$$
		Hence $\tilde{\mathbf\Gamma}_{\varepsilon,\tilde\phi}(\xi)$ is a $C^1$ closed curve. The quantitative estimates \eqref{2-15} and \eqref{2-16} also follows from the implicit function theorem directly.
	\end{proof}
	
	In view of Lemma \ref{lem2-3}, we consider
	\begin{equation*}
		\phi_\varepsilon=\mathcal T_\varepsilon N_\varepsilon(\phi_\varepsilon).
	\end{equation*}
	In the next lemma, we show the smallness and the contraction property for the error term $N_\varepsilon(\phi_\varepsilon)$, 
	so that the contraction mapping theorem is applied to obtain the existence of $\phi_\varepsilon$ in $E_\varepsilon$.
	\begin{lemma}\label{lem2-5}
		There exists a small $\varepsilon_0>0$ such that for any $\varepsilon\in(0,\varepsilon_0]$, there is a unique solution $\phi_\varepsilon\in E_\varepsilon$ to \eqref{2-6}. Moreover  $\phi_\varepsilon$ satisfies
		\begin{equation}\label{2-17}
			\|\phi_\varepsilon\|_*+\varepsilon^{1-\frac{2}{p}}\|\nabla\phi_\varepsilon\|_{L^p(B_{L\varepsilon}(\boldsymbol z_1^+)\cup B_{L\varepsilon}(\boldsymbol z_1^-))}= O(\varepsilon)
		\end{equation}
		for $p\in(2,+\infty]$.
	\end{lemma}
	\begin{proof}
		Denote $\mathcal G_\varepsilon:=\mathcal T_\varepsilon R_\varepsilon$, and a neighborhood of origin in $E_\varepsilon$ as
		\begin{equation*}
			\mathcal B_\varepsilon:=E_\varepsilon\cap \left\{\phi \ | \ \|\phi\|_*+\varepsilon^{1-\frac{2}{p}}\|\nabla\phi\|_{L^p(B_{L\varepsilon}(\boldsymbol z_1^+)\cup B_{L\varepsilon}(\boldsymbol z_1^-))}\le C\varepsilon, \ p\in(2,\infty]\right\}
		\end{equation*}
		with $C$ a large positive constant. We will show that $\mathcal G_\varepsilon$ is a contraction map from $\mathcal B_\varepsilon$ to $\mathcal B_\varepsilon$, so that a unique fixed point $\phi_\varepsilon$ can be obtained by the contraction mapping theorem. Actually, letting $\mathbf h=N_\varepsilon(\phi)$ for $\phi\in\mathcal B_\varepsilon$, and noticing that $N_\varepsilon(\phi)$ satisfies the assumptions for $\mathbf h$ in Lemma \ref{lem2-4}, we obtain
		\begin{equation*}
			\|\mathcal T_\varepsilon N_\varepsilon(\phi)\|_*+\varepsilon^{1-\frac{2}{p}}\|\nabla\mathcal T_\varepsilon N_\varepsilon(\phi)\|_{L^p(B_{L\varepsilon}(\boldsymbol z_1^+)\cup B_{L\varepsilon}(\boldsymbol z_1^-))}\le c_0\varepsilon^{1-\frac{2}{p}}\|N_\varepsilon(\phi) \|_{W^{-1,p}(B_{L\varepsilon}(\boldsymbol z_1^+)\cup B_{L\varepsilon}(\boldsymbol z_1^-))}.
		\end{equation*}
		
		To begin with, we are to show that $\mathcal G_\varepsilon$ maps $\mathcal B_\varepsilon$ continuously into itself. For simplicity, we only consider $N_\varepsilon^+(\phi)$ restricted in $\Pi^+$. For each $\zeta(\boldsymbol y)\in C_0^\infty(B_{L}(\boldsymbol 0))$, by applying the estimates in Lemma \ref{lem2-4}, we derive
		\begin{align*}
				&\quad\frac{s_\varepsilon^2}{\varepsilon^2}\int_{B_L(\boldsymbol 0)}\left[\boldsymbol1_{\{\tilde\psi_\varepsilon+W_\ep\cos(s_\varepsilon y_1+\theta_0)>\mu_\varepsilon\}}-\boldsymbol 1_{\{\tilde V_{1,\varepsilon}^+>\frac{\kappa}{2\pi}\ln\frac{1}{\varepsilon}\}}\right]\zeta d\boldsymbol y-2\int_0^{2\pi}\tilde\phi \zeta(1,\xi)d\xi\\
				&=\frac{s_\varepsilon^2}{\varepsilon^2}\int_0^{2\pi}\int_1^{1+t_\varepsilon(\xi)}t\zeta(t,\xi)dtd\xi-2\int_0^{2\pi}\tilde\phi\zeta(1,\xi)d\xi\\
				&=\frac{s_\varepsilon^2}{\varepsilon^2}\int_0^{2\pi}\int_1^{1+t_\varepsilon(\xi)}t\zeta(1,\xi)dtd\xi-2\int_0^{2\pi}\tilde\phi\zeta(1,\xi)d\xi\\
				& \ \ \ +\frac{s_\varepsilon^2}{\varepsilon^2} \int_0^{2\pi}\int_1^{1+t_\varepsilon(\xi)}t(\zeta(t,\theta)-\zeta(1,\xi))dtd\xi\\
				&=\frac{s_\varepsilon^2}{\varepsilon^2}\int_0^{2\pi}\frac{\tilde\phi(1,\xi)}{s_\varepsilon\beta_\varepsilon} \zeta(1,\xi)d\xi-2\int_0^{2\pi}\tilde\phi\zeta(1,\xi)d\xi\\
				& \ \ \ +\frac{s_\varepsilon^2}{\varepsilon^2}\int_0^{2\pi}\int_1^{1+t_\varepsilon(\xi)}t\int_1^t\frac{\partial \tilde\varphi(r,\xi)}{\partial r}drdtd\xi=O(\varepsilon)\|\zeta\|_{W^{1,p'}(B_{L}(\boldsymbol 0))},
		\end{align*}
		where we have used the estimate for the gradient value $\beta_\varepsilon$ in \eqref{2-1}. Hence it holds
		\begin{equation*}
			\varepsilon^{1-\frac{2}{p}}\|N_\varepsilon(\phi)\|_{W^{-1,p}(B_{L\varepsilon}(\boldsymbol z_1^+)\cup B_{L\varepsilon}(\boldsymbol z_1^-))}=O(\varepsilon),
		\end{equation*}
		and
		\begin{equation*}
			\|\mathcal T_\varepsilon N_\varepsilon(\phi)\|_*+\varepsilon^{1-\frac{2}{p}}\|\nabla\mathcal T_\varepsilon N_\varepsilon(\phi)\|_{L^p(B_{L\varepsilon}(\boldsymbol z_1^+)\cup B_{L\varepsilon}(\boldsymbol z_1^-))}=O(\varepsilon)
		\end{equation*}
		for $p\in(2,+\infty]$. Thus the operator $\mathcal G_\varepsilon$ indeed maps $\mathcal B_\varepsilon$ to $\mathcal B_\varepsilon$ continuously.

		In the next step, we are to verify that $\mathcal G_\varepsilon$ is a contraction mapping under the norm
		\begin{equation*}
			\|\cdot\|_{\mathcal G_\varepsilon}=\|\cdot\|_*+\varepsilon^{1-\frac{2}{p}}\|\cdot\|_{W^{1,p}(B_{L\varepsilon}(\boldsymbol z_1^+)\cup B_{L\varepsilon}(\boldsymbol z_1^-))}, \quad p\in(2,+\infty].
		\end{equation*}
		We already know that $\mathcal B_\varepsilon$ is close under this norm. Let $\phi_1$ and $\phi_2$ be two functions in $\mathcal B_\varepsilon$. By Lemma \ref{lem2-2} and symmetry of the problem with respect to $\theta=\pi/2$, we have
		\begin{equation*}
			\|\mathcal G_\varepsilon\phi_1-\mathcal G_\varepsilon\phi_2\|_{\mathcal G_\varepsilon}\le C\varepsilon^{1-\frac{2}{p}}\|N_\varepsilon^+(\phi_1)-N_\varepsilon^+(\phi_2) \|_{W^{-1,p}(B_{L\varepsilon}(\boldsymbol z_1^+))},
		\end{equation*}
		where
		\begin{align*}
				N_\varepsilon^+(\phi_1)&-N_\varepsilon^+(\phi_2)\\
				&=\frac{1}{\varepsilon^2}\left(\boldsymbol 1_{\{\Psi_\varepsilon+\phi_1+W_\ep\cos\theta>\mu_\varepsilon\}}-\boldsymbol1_{\{\Psi_\varepsilon+\phi_2+W_\ep\cos\theta>\mu_\ep\}}-2\left[\phi_1(r,\xi)-\phi_2(r,\xi)\right]\boldsymbol\delta_{\mathcal C_{s_\varepsilon}(\boldsymbol z_1^+)}\right).
		\end{align*}
		According to \eqref{2-16} in Lemma \ref{lem2-4}, for each $\zeta\in C_0^\infty(B_L(\boldsymbol 0))$, it holds
		\begin{align*}
				&\quad\,\, \frac{s_\varepsilon^2}{\varepsilon^2}\int_{B_L(\boldsymbol 0)}\left[\boldsymbol1_{\{\{\tilde\Psi_\varepsilon+\tilde\phi_1+W_\ep\cos(s_\varepsilon y_1+\theta_0)>\mu_\varepsilon\}\}\}}-\boldsymbol1_{\{\tilde\Psi_\varepsilon+\tilde\phi_1+W_\ep\cos(s_\varepsilon y_1+\theta_0)>\mu_\varepsilon\}}\right]\zeta d\boldsymbol y\\
				&=\frac{s_\varepsilon^2}{\varepsilon^2}\int_0^{2\pi}(t_{\varepsilon,\tilde\phi_1}-t_{\varepsilon,\tilde\phi_2})\zeta(1,\xi)d\xi+o_\varepsilon(1)\|\tilde\phi_1-\tilde\phi_2\|_{L^\infty(B_L(\boldsymbol 0))}\|\zeta\|_{W^{1,p'}(B_L(\boldsymbol 0))}.
		\end{align*}
		By denoting $\phi^*=\phi_1-\phi_2$, we then obtain the expansion
		\begin{align*}
				&\quad t_{\varepsilon,\tilde\phi_1}-t_{\varepsilon,\tilde\phi_2}\\
				&=-s_\varepsilon\beta_\varepsilon\left(\phi_*(1,\xi)+\int_1^{1+t_{\varepsilon,1}(\xi)}\frac{\partial\tilde\phi_*(r,\xi)}{\partial r}dr+\int_{1+t_{\varepsilon,1}(\xi)}^{1+t_{\varepsilon,2}(\xi)}\frac{\partial\tilde\phi_2(r,\xi)}{\partial r}dr\right).
		\end{align*}
		Then using the definition of $\mathcal \beta_\varepsilon$ in \eqref{2-1}, one can deduce that
		\begin{align*}
				&\quad\frac{s_\varepsilon^2}{\varepsilon^2}\int_0^{2\pi}(t_{\varepsilon,\tilde\phi_1}-t_{\varepsilon,\tilde\phi_2})\zeta(1,\xi)d\xi\\
				&=2\int_0^{2\pi}(\tilde\phi_1-\tilde\phi_2)\zeta(1,\xi)d\xi+o_\varepsilon(1)\|\tilde\phi^*\|_{L^\infty(B_L(\boldsymbol 0))}\|\zeta\|_{W^{1,p'}(B_L(\boldsymbol 0))}.
		\end{align*}
		Finally, we conclude that
		\begin{equation*}
			\varepsilon^{1-\frac{2}{p}}\|N^+_\varepsilon(\phi_1)-N^+_\varepsilon(\phi_2) \|_{W^{-1,p}(B_{L\ep}(\boldsymbol z_1^+))}=o_\varepsilon(1) \|\phi_1-\phi_2\|_{\mathcal G_\varepsilon},
		\end{equation*}
		and
		\begin{equation*}
			\|\mathcal G_\varepsilon\phi_1-\mathcal G_\varepsilon\phi_2\|_{\mathcal G_\varepsilon}=o_\varepsilon(1) \|\phi_1-\phi_2\|_{\mathcal G_\varepsilon}.
		\end{equation*}
		Hence we have shown that $\mathcal G_\varepsilon$ is a contraction map from $\mathcal B_\varepsilon$ into itself.
		
		By applying the contraction mapping theorem, we now can claim that there is a unique $\phi_\varepsilon\in \mathcal B_\varepsilon$ such that $\phi_\varepsilon=\mathcal G_\varepsilon\phi_\varepsilon$, which satisfies \eqref{2-17}. 
	\end{proof}
	
	Although we obtain a solution to the projective problem \eqref{2-6}, it does not solve the primitive equation \eqref{2-4}. To go back to the 
	problem \eqref{2-4}, we will solve a one-dimensional problem in the next part.
	
	\subsection{The one-dimensional problem}
	
	From \eqref{2-6}, it holds
	\begin{equation*}
		(-\Delta_{\mathbb S^2})\psi_\varepsilon-\frac{1}{\varepsilon^2}\boldsymbol1_{\{\psi_\varepsilon+W_\varepsilon\cos\theta>\mu_\varepsilon\}}+\frac{1}{\varepsilon^2}\boldsymbol1_{\{-\psi_\varepsilon-W_\varepsilon\cos\theta>\mu_\varepsilon\}}=\Lambda(-\Delta_{\mathbb S^2})Z_\varepsilon.
	\end{equation*}
	Multiplying above equality by $Z_\varepsilon$ and integrating on $\Pi$, we have
	\begin{equation*}
		\int_\Pi\left[(-\Delta_{\mathbb S^2})\psi_\varepsilon-\frac{1}{\varepsilon^2}\boldsymbol1_{\{\psi_\varepsilon+W_\varepsilon\cos\theta>\mu_\varepsilon\}}+\frac{1}{\varepsilon^2}\boldsymbol1_{\{-\psi_\varepsilon-W_\varepsilon\cos\theta>\mu_\varepsilon\}}\right]Z_\varepsilon d\boldsymbol\sigma=\Lambda\int_\Pi Z_\varepsilon(-\Delta_{\mathbb S^2})Z_\varepsilon d\boldsymbol \sigma.
	\end{equation*}
	To make $\psi_\varepsilon$ a solution to \eqref{2-4}, we choose a suitable traveling angular velocity $W_\varepsilon$ such that $\Lambda=0$, which needs the following lemma.
	
	\begin{lemma}\label{lem2-6}
		It holds
		\begin{align*}
			&\quad\int_\Pi\left[(-\Delta_{\mathbb S^2})\psi_\varepsilon-\frac{1}{\varepsilon^2}\boldsymbol1_{\{\psi_\varepsilon+W_\varepsilon\cos\theta>\mu_\varepsilon\}}+\frac{1}{\varepsilon^2}\boldsymbol1_{\{-\psi_\varepsilon-W_\varepsilon\cos\theta>\mu_\varepsilon\}}\right]Z_\varepsilon d\boldsymbol\sigma\\
			&=2\kappa\left[\kappa\sum\limits_{i=2}^k\partial_\theta G(\boldsymbol z_1^+,\boldsymbol z_i^+)+ \kappa \partial_\theta H(\boldsymbol z_1^+,\boldsymbol z_1^+)-\kappa\sum\limits_{i=1}^k\partial_\theta G(\boldsymbol z_1^+,\boldsymbol z_i^-)-W_\varepsilon\sin\theta_0+o_\varepsilon(1)\right].
		\end{align*}
	\end{lemma}
	\begin{proof}
		Notice that
		\begin{equation*}
			(-\Delta_{\mathbb S^2})\psi_\varepsilon-\frac{1}{\varepsilon^2}\boldsymbol1_{\{\psi_\varepsilon+W_\varepsilon\cos\theta>\mu_\varepsilon\}}+\frac{1}{\varepsilon^2}\boldsymbol1_{\{-\psi_\varepsilon-W_\varepsilon\cos\theta>\mu_\varepsilon\}}=\mathbb L_\varepsilon\phi_\varepsilon-N_\varepsilon(\phi_\varepsilon).
		\end{equation*}
		For simplicity, we consider the following integration in $\Pi^+=(0,\pi/2)\times(0,2\pi/k)$.
		\begin{equation*}
			\int_{\Pi^+}Z_\varepsilon\mathbb L_\varepsilon\phi_\varepsilon d\boldsymbol \sigma-\int_{\Pi^+}Z_\varepsilon N_\varepsilon^+(\phi_\varepsilon)d\boldsymbol \sigma=I_1-I_2.
		\end{equation*}
		
		For the first part of the integration $I_1$, from Lemma \ref{lem2-2}, we already know that 
		$$I_1=\frac{C}{\varepsilon|\ln\varepsilon|}\|\phi_\varepsilon\|_*=o_\varepsilon(1).$$ 
		For the second part of the integration $I_2$, by Lemma \ref{lem2-4} it holds 
			\begin{align*}
			\int_{\Pi^+}Z_\varepsilon N_\varepsilon^+(\phi_\varepsilon)d\boldsymbol \sigma&=\frac{1}{\varepsilon^2}\left[\boldsymbol1_{\{\tilde\psi_\varepsilon+W\cos(s_\varepsilon y_1+\theta_0)>\mu_\varepsilon\}}-\boldsymbol 1_{\{\tilde V_{1,\varepsilon}^+>\frac{\kappa}{2\pi}\ln\frac{1}{\varepsilon}\}}-\frac{2}{s_\varepsilon}\phi_\varepsilon\boldsymbol\delta_{\mathcal C_{s_\varepsilon}(\boldsymbol z_1^+)}\right]Z_\varepsilon d\boldsymbol\sigma\\
			&=\frac{s_\varepsilon^2}{\varepsilon^2}\int_0^{2\pi}\left[\frac{\mathcal H(\xi)}{s_\varepsilon\beta_\varepsilon}+O(\varepsilon^2)\right]Z_\varepsilon d\xi+o_\varepsilon(1),
		\end{align*}
		where $\beta_\ep$ is defined in \eqref{2-1}. Then we have 
		\begin{align*}
			I_2&=\frac{s_\varepsilon^2}{\varepsilon^2}\int_0^{2\pi}\frac{\mathcal H(\xi)}{s_\varepsilon\beta_\varepsilon} Z_\varepsilon d\xi+o_\varepsilon(1)\\
			&=-\kappa\left[\kappa\sum\limits_{i=2}^k\partial_\theta G(\boldsymbol z_1^+,\boldsymbol z_i^+)+ \kappa \partial_\theta H(\boldsymbol z_1^+,\boldsymbol z_1^+)-\kappa\sum\limits_{i=1}^k\partial_\theta G(\boldsymbol z_1^+,\boldsymbol z_i^-)-W_\varepsilon\sin\theta_0\right],
		\end{align*}
		where we have used the asymptotic estimates $\frac{s_\varepsilon^2}{\varepsilon^2}=\frac{\kappa}{\pi}+o_{\varepsilon}(1)$ and $s_\varepsilon\beta_\varepsilon=\frac{\kappa}{2\pi}+o_\varepsilon(1)$ in \eqref{2-1}. 
		Since the integration on $\Pi^-=(\pi/2,\pi)\times(0,2\pi/k)$ is the same as in $\Pi^+$, the proof is complete by multiplying it twice.
	\end{proof}

	\noindent Now we are ready to give the proof for Theorem \ref{thm1}.
	
	\noindent{\bf Proof of Theorem \ref{thm1}:} To obtain a family of desired solutions to \eqref{2-4}, we only need to find suitable $W_\varepsilon$ such that coefficient $\Lambda=0$. In the proof of Lemma \ref{lem2-6}, we already obtain
	\begin{equation*}
		\int_\Pi Z_\varepsilon(-\Delta_{\mathbb S^2})Z_\varepsilon d\boldsymbol \sigma=\int_{\Pi} \left[(\partial_\theta Z_\varepsilon)^2+\left(\frac{\partial_\varphi Z_\varepsilon}{\sin\theta}\right)^2\right]d\boldsymbol \sigma>0.
	\end{equation*} 
	Hence by Lemma \ref{lem2-6}, $\Lambda=0$ is equivalent to the characterization
	\begin{equation*}
		W_\varepsilon=\frac{\kappa}{\sin\theta_0}\sum\limits_{i=2}^k\partial_\theta G(\boldsymbol z_1^+,\boldsymbol z_i^+)+\frac{\kappa}{\sin\theta_0}\partial_\theta H(\boldsymbol z_1^+,\boldsymbol z_1^+)-\frac{\kappa}{\sin\theta_0}\sum\limits_{i=1}^k\partial_\theta G(\boldsymbol z_1^+,\boldsymbol z_i^-)+o_\varepsilon(1).
	\end{equation*}
	As a result, for the $(\mathbf{type\, 1})$ vortex streets, the traveling angular velocity is
	\begin{align*}
		W_\varepsilon^{(1)}=&\frac{\kappa}{\sin\theta_0}\sum\limits_{i=2}^k\partial_\theta G\left(\theta_0,\frac{\pi}{k},\theta_0,\frac{2\pi i}{k}-\frac{\pi}{k}\right)+\frac{\kappa}{\sin\theta_0}\partial_\theta H\left(\theta_0,\frac{\pi}{k},\theta_0,\frac{\pi }{k}\right)\\
		&-\frac{\kappa}{\sin\theta_0}\sum\limits_{i=1}^k\partial_\theta G\left(\theta_0,\frac{\pi}{k},\pi-\theta_0,\frac{2\pi i}{k}-\frac{\pi}{k}\right)+o_\varepsilon(1),
	\end{align*}
	and for the $(\mathbf{type\, 2})$ vortex streets, the traveling angular velocity is
   	\begin{align*}
   		W_\varepsilon^{(2)}=&\frac{\kappa}{\sin\theta_0}\sum\limits_{i=2}^k\partial_\theta G\left(\theta_0,\frac{\pi}{2k},\theta_0,\frac{2\pi i}{k}-\frac{3\pi}{2k}\right)+\frac{\kappa}{\sin\theta_0}\partial_\theta H\left(\theta_0,\frac{\pi}{2k},\theta_0,\frac{\pi i}{2k}\right)\\
   		&-\frac{\kappa}{\sin\theta_0}\sum\limits_{i=1}^k\partial_\theta G\left(\theta_0,\frac{\pi}{2k},\pi-\theta_0,\frac{2\pi i}{k}-\frac{\pi}{2k}\right)+o_\varepsilon(1),
   	\end{align*}
	which gives rise to the property (ii). The convergence for $\omega_\varepsilon=(-\Delta_{\mathbb S^2})\psi_\varepsilon$ stated in (i) is obvious. Finally, the estimate (iii) for vorticity sets $\Omega_{i,\varepsilon}^\pm$ is directly from Lemma \ref{lem2-4}, and $s_\varepsilon=\sqrt{\kappa/\pi} \varepsilon+o(\varepsilon)$ by \eqref{2-1}. Thus we have completed the proof.	\qed
	
	\section{The general steady case}\label{sec3}
	
	In this section, we discuss the general steady case where the vorticity solutions are near the vortex-wave system
	\begin{equation*}
		\omega^*(\boldsymbol z)=\sum\limits_{m=1}^{j}\kappa_m^+\boldsymbol \delta_{\boldsymbol z_m^+}-\sum\limits_{n=1}^{k}\kappa_n^-\boldsymbol \delta_{\boldsymbol z_n^-}+2\gamma\cos\theta,
	\end{equation*}
	with the rotating speed of the sphere $\gamma$ and $\sum\limits_{m=1}^{j}\kappa_m^+=\sum\limits_{n=1}^{k}\kappa_n^-$.
	To approximate the singular part for the stream function of each vortex, 
	let $\boldsymbol z^+_{m,\varepsilon}=(\phi^+_{m,\varepsilon}, \theta^+_{m,\varepsilon})$, $\boldsymbol z^-_{n,\varepsilon}=(\phi^-_{n,\varepsilon}, \theta^-_{n,\varepsilon})$ be undetermined, and define  
	\begin{equation*}
		V^+_{m,\varepsilon}(\boldsymbol z)=\left\{
		\begin{array}{lll}
			\frac{\kappa_m^+}{2\pi}\ln\frac{1}{\varepsilon}+\frac{1}{4\varepsilon^2}(s_{m,\varepsilon}^{+2}-|A(\boldsymbol z-\boldsymbol z^+_{m,\varepsilon})|^2), \ \ \ &\mathrm{if} \ |A(\boldsymbol z-\boldsymbol z^+_{m,\varepsilon})|\le s_{m,\varepsilon}^+,\\
			\frac{\kappa_m^+}{2\pi}\frac{|\ln\varepsilon|}{|\ln s_{m,\varepsilon}^+|}\ln|A(\boldsymbol z-\boldsymbol z^+_{m,\varepsilon})|,&\mathrm{if} \ |A(\boldsymbol z-\boldsymbol z^+_{m,\varepsilon})|\ge s_{m,\varepsilon}^+
		\end{array}
		\right.
	\end{equation*}
	and
	\begin{equation*}
		V^-_{n,\varepsilon}(\boldsymbol z)=\left\{
		\begin{array}{lll}
			\frac{\kappa_n^-}{2\pi}\ln\frac{1}{\varepsilon}+\frac{1}{4\varepsilon^2}(s_{n,\varepsilon}^{-2}-|A(\boldsymbol z-\boldsymbol z^-_{n,\varepsilon})|^2), \ \ \ &\mathrm{if} \ |A(\boldsymbol z-\boldsymbol z^-_{n,\varepsilon})|\le s_{n,\varepsilon}^-,\\
			\frac{\kappa_n^-}{2\pi}\frac{|\ln\varepsilon|}{|\ln s_{n,\varepsilon}^-|}\ln|A(\boldsymbol z-\boldsymbol z^-_{n,\varepsilon})|,&\mathrm{if} \ |A(\boldsymbol z-\boldsymbol z^-_{n,\varepsilon})|\ge s_{n,\varepsilon}^-,
		\end{array}
		\right.
	\end{equation*}
	where $A$ is the same tangent mapping defined in Section \ref{sec2}, and $s_{m,\varepsilon}^+,s_{n,\varepsilon}^-$ satisfy the regularity condition
	\begin{equation*}
		\beta_{m,\varepsilon}^+=\frac{\kappa_m^+}{2\pi}\frac{|\ln\varepsilon|}{s_{m,\varepsilon}^+|\ln s_{m,\varepsilon}^+|}=\frac{s_{m,\varepsilon}^+}{2\varepsilon^2}, \quad \beta_{n,\varepsilon}^-=\frac{\kappa_n^-}{2\pi}\frac{|\ln\varepsilon|}{s_{n,\varepsilon}^-|\ln s_{n,\varepsilon}^-|}=\frac{s_{n,\varepsilon}^-}{2\varepsilon^2}.
	\end{equation*}
	Then let
	\begin{equation*}
		R^+_{m,\varepsilon}(\boldsymbol z)=\frac{1}{\varepsilon^2}\int_{\{V^+_{m,\varepsilon}(\boldsymbol z)>\frac{\kappa}{2\pi}\ln\frac{1}{\varepsilon}\}} H(\theta,\varphi,\theta',\varphi')d\boldsymbol \sigma(\boldsymbol z'), 
	\end{equation*}
	and
	\begin{equation*}	
		\quad R^-_{n,\varepsilon}(\boldsymbol z)=\frac{1}{\varepsilon^2}\int_{\{V^-_{n,\varepsilon}(\boldsymbol z)>\frac{\kappa}{2\pi}\ln\frac{1}{\varepsilon}\}} H(\theta,\varphi,\theta',\varphi')d\boldsymbol \sigma(\boldsymbol z')
	\end{equation*}
	be the approximations of the regular part $H$. The solution $\psi_\varepsilon$ to \eqref{1-8} is decomposed as
	\begin{align*}
		\psi_\varepsilon(\boldsymbol z)&=\sum\limits_{m=1}^jU^+_{m,\varepsilon}+\sum\limits_{m=1}^jR^+_{m,\varepsilon}-\sum\limits_{n=1}^kU^-_{n,\varepsilon}-\sum\limits_{n=1}^kR^-_{n,\varepsilon}+\phi_\varepsilon\\
		&:=\Psi_\varepsilon+\phi_\varepsilon, 
    \end{align*}	
	where $\phi_\varepsilon$ is a small perturbation term. Using the linearization method, we then transform equation \eqref{1-6} to a semilinear problem on $\phi_\varepsilon$
	\begin{equation}\label{3-1}
		\mathbb L_\varepsilon\phi_\varepsilon=N_\varepsilon(\phi_\varepsilon),
	\end{equation}
	where
	\begin{equation*}
	\mathbb L_\varepsilon\phi_\varepsilon=(-\Delta_{\mathbb S^2})\phi_\varepsilon-\sum\limits_{m=1}^j\frac{2}{s_{m,\varepsilon}^+}\phi_\varepsilon\boldsymbol\delta_{\mathcal C_{s_{m,\varepsilon}^+}(\boldsymbol z_{m,\varepsilon}^+)}-\sum\limits_{n=1}^k\frac{2}{s_{n,\varepsilon}^+}\phi_\varepsilon\boldsymbol\delta_{\mathcal C_{s_{n,\varepsilon}^-}(\boldsymbol z_{n,\varepsilon}^-)},
	\end{equation*}
	and
	\begin{align*}
	N_\varepsilon(\phi_\varepsilon)=&\frac{1}{\varepsilon^2}\sum\limits_{m=1}^j\bigg(\boldsymbol1_{ B_\delta(\boldsymbol z_m^+)}\boldsymbol1_{\{\psi_\varepsilon+\gamma\cos\theta>\mu_{m,\varepsilon}^+\}}-\boldsymbol1_{\{V_{m,\varepsilon}^+>\frac{\kappa_m^+}{2\pi}\ln\frac{1}{\varepsilon}\}}-\frac{2}{s_{m,\varepsilon}^+}\phi_\varepsilon\boldsymbol\delta_{\mathcal C_{s_{m,\varepsilon}^+}(\boldsymbol z_{m,\varepsilon}^+)}\bigg)\\
	&-\frac{1}{\varepsilon^2}\sum\limits_{n=1}^k\bigg(\boldsymbol1_{B_\delta(\boldsymbol z_n^-)}\boldsymbol1_{\{-\psi_\varepsilon-\gamma\cos\theta>\mu_{n,\varepsilon}^-\}}-\boldsymbol1_{\{V_{n,\varepsilon}^->\frac{\kappa_n^-}{2\pi}\ln\frac{1}{\varepsilon}\}}+\frac{2}{s_{n,\varepsilon}^-}\phi_\varepsilon\boldsymbol\delta_{\mathcal C_{s_{n,\varepsilon}^-}(\boldsymbol z_{n,\varepsilon}^-)}\bigg)
	\end{align*}
	with 
	$$\mathcal C_{s_{m,\varepsilon}^+}(\boldsymbol z_{m,\varepsilon}^+)=\{\boldsymbol z\in \mathbb S^2\mid |A(\boldsymbol z-\boldsymbol z^+_{m,\varepsilon})|= s_{m,\varepsilon}^+\},$$
	and
	$$\mathcal C_{s_{n,\varepsilon}^-}(\boldsymbol z_{n,\varepsilon}^-)=\{\boldsymbol z\in \mathbb S^2\mid |A(\boldsymbol z-\boldsymbol z^-_{n,\varepsilon})|= s_{n,\varepsilon}^-\}.$$ 
	To make $N_\varepsilon(\phi_\varepsilon)$ as small as possible, $\mu_{m,\varepsilon}^+$ and $\mu_{n,\varepsilon}^-$ are chosen to satisfy
	\begin{align*}
		-\frac{\kappa_m^+}{2\pi}\ln\frac{1}{\varepsilon}=&-\sum\limits_{i\neq m}^j\kappa_i^+G(\boldsymbol z_{m,\varepsilon}^+,\boldsymbol z_{i,\varepsilon}^+)-\sum\limits_{l=1}^k\kappa_l^-G(\boldsymbol z_{m,\varepsilon}^+,\boldsymbol z_{l,\varepsilon}^-)\\
		&-\kappa_m^+ H(\boldsymbol z_{m,\varepsilon}^+,\boldsymbol z_{m,\varepsilon}^+)-\gamma\cos\theta_{m,\varepsilon}^++\mu_{m,\varepsilon}^+,
	\end{align*}
	and
	\begin{align*}
		-\frac{\kappa_n^-}{2\pi}\ln\frac{1}{\varepsilon}=&-\sum\limits_{l\neq n}^j\kappa_l^-G(\boldsymbol z_{n,\varepsilon}^-,\boldsymbol z_{l,\varepsilon}^-)-\sum\limits_{i=1}^k\kappa_i^+G(\boldsymbol z_{n,\varepsilon}^-,\boldsymbol z_{i,\varepsilon}^+)\\
		&- \kappa_n^-H(\boldsymbol z_{n,\varepsilon}^-,\boldsymbol z_{n,\varepsilon}^-)-\gamma\cos\theta_{n,\varepsilon}^-+\mu_{n,\varepsilon}^-.
	\end{align*}
	
   \subsection{The reduction}
   
   Since $\mathbb L_\varepsilon$ is not invertible as we discussed in Section \ref{sec2}, we first consider the projective problem of \eqref{3-1}. For this purpose, we define the tangent mapping $A_l^\pm: (\theta,\varphi)\to (x_1,x_2)$ from $\mathbb S^2$ to $T_{\boldsymbol z_l^\pm}\mathbb S^2$ with the matrix 
   \begin{equation*}
   	\mathrm{Mat}(A_l^\pm)=\left(
   	\begin{array}{ccc}
   		1 & 0             \\
   		0 & \sin\theta_{l,\varepsilon}^\pm
   	\end{array}
   	\right).
   \end{equation*}
   We also define $U^\pm_{l,\varepsilon}(\boldsymbol z)$ as 
   \begin{equation*}
   	U^\pm_{l,\varepsilon}(\boldsymbol z)=\left\{
   	\begin{array}{lll}
   		\frac{\kappa_l^\pm}{2\pi}\ln\frac{1}{\varepsilon}+\frac{1}{4\varepsilon^2}(s_{l,\varepsilon}^{\pm2}-|A_l^{\pm}(\boldsymbol z-\boldsymbol z^\pm_{l,\varepsilon})|^2), \ \ \ &\mathrm{if} \ |A_l^{\pm}(\boldsymbol z-\boldsymbol z^\pm_{l,\varepsilon})|\le s_{l,\varepsilon}^\pm,\\
   		\frac{\kappa_l^\pm}{2\pi}\frac{|\ln\varepsilon|}{|\ln s_{l,\varepsilon}^\pm|}\ln|A_l^{\pm}(\boldsymbol z-\boldsymbol z^\pm_{l,\varepsilon})|,&\mathrm{if} \ |A_l^{\pm}(\boldsymbol z-\boldsymbol z^\pm_{l,\varepsilon})|\ge s_{l,\varepsilon}^\pm.
   	\end{array}
   	\right.
   \end{equation*}
   Then the approximate kernel of $\mathbb L_\varepsilon$ is $(2j+2k)$-dimensional, which is given by the linear combination of
   \begin{equation*}
   	X_{m,\varepsilon}^+({\boldsymbol z})=\chi_m^+({\boldsymbol z})\frac{\partial U^+_{m,\varepsilon}}{\partial \theta}, 
	\quad Y_{m,\varepsilon}^+({\boldsymbol z})=\chi_m^+({\boldsymbol z})\frac{\partial U^+_{m,\varepsilon}}{\partial \varphi}, \quad 1\le m\le j,
   \end{equation*}
   and
   \begin{equation*}
   	X_{n,\varepsilon}^-({\boldsymbol z})=\chi_n^-({\boldsymbol z})\frac{\partial U^+_{n,\varepsilon}}{\partial \theta}, 
	\quad Y_{n,\varepsilon}^-({\boldsymbol z})=\chi_n^-({\boldsymbol z})\frac{\partial U^+_{n,\varepsilon}}{\partial \varphi}, \quad 1\le n\le k,
   \end{equation*}
   where
   \begin{equation*}
   	\chi_l^\pm(\boldsymbol z)=\left\{
   	\begin{array}{lll}
   		1, \ \ \  & \mathrm{if} \ |\boldsymbol z-\boldsymbol z_{l,\varepsilon}^\pm|_{\mathbb S^2}< \varepsilon|\ln\varepsilon|,\\
   		0, & \mathrm{if} \ |\boldsymbol z-\boldsymbol z_{l,\varepsilon}^\pm|_{\mathbb S^2}\ge 2\varepsilon|\ln\varepsilon|
   	\end{array}
   	\right.
   \end{equation*}
   are smooth truncation functions radially symmetric with respect to $\boldsymbol z_{l,\ep}^\pm$ satisfying
   \begin{equation*}
   	|\nabla \chi_i^\pm(\boldsymbol z)|\le \frac{2}{\varepsilon|\ln\varepsilon|} \quad \mathrm{and} \quad |\nabla^2 \chi_i^\pm(\boldsymbol z)|\le \frac{2}{\varepsilon^2|\ln\varepsilon|^2}.
   \end{equation*}
   
   With these preparations, the projection problem for \eqref{3-1} is written as
   \begin{equation}\label{3-2}
   	\begin{cases}
   		\mathbb L_\varepsilon\phi=\mathbf h(\boldsymbol z)+(-\Delta_{\mathbb S^2})\sum_{m=1}^j\left[a_m^+X_{m,\varepsilon}^++b_m^+Y_{m,\varepsilon}^+\right]\\
   		\quad\quad\quad\quad\quad\ +(-\Delta_{\mathbb S^2})\sum_{n=1}^k\left[a_n^-X_{n,\varepsilon}^-+b_n^-Y_{n,\varepsilon}^-\right], \ \ &\text{in} \ \mathbb S^2,\\
   		\int_{\mathbb S^2}  \phi(\boldsymbol z)(-\Delta_{\mathbb S^2})X_{l,\varepsilon}^\pm(\boldsymbol z) d\boldsymbol \sigma=0, \quad
   		\int_{\mathbb S^2}  \phi(\boldsymbol z)(-\Delta_{\mathbb S^2})Y_{i,\varepsilon}^\pm(\boldsymbol z) d\boldsymbol \sigma=0,
   	\end{cases}
   \end{equation}
   where the nonlinear term $\mathbf h(\boldsymbol z)$ satisfies 
   $$\mathrm{supp}\, \mathbf h(\boldsymbol z)\subset \left(\cup_{m=1}^jB_{L\varepsilon}(\boldsymbol z_{m,\varepsilon}^+)\right)\cup \left(\cup_{n=1}^kB_{L\varepsilon}(\boldsymbol z_{n,\varepsilon}^-)\right)$$
   with $L$ a large positive constant, and 
   $$\mathbf\Lambda=(a_1^+,\cdots,a_j^+,b_1^+,\cdots,b_j^+, a_1^-,\cdots,a_k^-,b_1^-,\cdots,b_k^-)$$
   is the $(2j+2k)$-dimensional projection vector determined by
   \begin{flalign*}
   	&\quad\quad\quad\quad a_l^\pm\int_{\mathbb S^2} X_{l,\varepsilon}^\pm(-\Delta_{\mathbb S^2})X_{l,\varepsilon}^\pm d\boldsymbol \sigma=\int_{\mathbb S^2}X_{l,\varepsilon}^\pm\big[\mathbb L_\ep\phi-\mathbf h(\boldsymbol z)\big]d\boldsymbol \sigma,&
   \end{flalign*}
   \begin{flalign*}
   	&\quad\quad\quad\quad b_l^\pm\int_{\mathbb S^2} Y_{l,\varepsilon}^\pm(-\Delta_{\mathbb S^2})Y_{l,\varepsilon}^\pm d\boldsymbol \sigma=\int_{\mathbb S^2}Y_{l,\varepsilon}^\pm\big[\mathbb L_\ep\phi-\mathbf h(\boldsymbol z)\big]d\boldsymbol \sigma.&
   \end{flalign*}
   By denoting the $\|\cdot\|_*$ norm for the function $\phi$ on the sphere as  
   $$\|\phi\|_*=\sup_{\mathbb S^2} |\phi|,$$
   we can prove the following coercive estimate for $\mathbb L_\varepsilon$ by a similar spirit of Lemma \ref{lem2-2}. 
   \begin{lemma}\label{lem3-1}
   	Assume that $\mathbf h$ satisfies $\mathrm{supp}\, \mathbf h\subset \left(\cup_{m=1}^jB_{L\varepsilon}(\boldsymbol z_{m,\varepsilon}^+)\right)\cup \left(\cup_{n=1}^kB_{L\varepsilon}(\boldsymbol z_{n,\varepsilon}^-)\right),$ and $$\varepsilon^{1-\frac{2}{p}}\|\mathbf h\|_{W^{-1,p}\left(\left(\cup_{m=1}^jB_{L\varepsilon}(\boldsymbol z_{m,\varepsilon}^+)\right)\cup \left(\cup_{n=1}^kB_{L\varepsilon}(\boldsymbol z_{n,\varepsilon}^-)\right)\right)}<\infty$$
   	for $p\in (2,+\infty]$, then there exists a small $\varepsilon_0>0$ and a positive constant $c_0$ such that for any $\varepsilon\in(0,\varepsilon_0]$ and solution pair $(\phi,\mathbf\Lambda)$ to \eqref{3-2}, it holds
   	\begin{align*}
   		\|\phi\|_*+\varepsilon^{1-\frac{2}{p}}\|\nabla\phi&\|_{L^p\left(\left(\cup_{m=1}^jB_{L\varepsilon}(\boldsymbol z_{m,\varepsilon}^+)\right)\cup \left(\cup_{n=1}^kB_{L\varepsilon}(\boldsymbol z_{n,\varepsilon}^-)\right)\right)}+|\mathbf\Lambda|\cdot\varepsilon^{-1}\\
   		&\le c_0\varepsilon^{1-\frac{2}{p}}\|\mathbf h\|_{W^{-1,p}\left(\left(\cup_{m=1}^jB_{L\varepsilon}(\boldsymbol z_{m,\varepsilon}^+)\right)\cup \left(\cup_{n=1}^kB_{L\varepsilon}(\boldsymbol z_{n,\varepsilon}^-)\right)\right)}.
   	\end{align*}
   \end{lemma}
   
   According to Lemma~\ref{lem3-1}, we use the contraction mapping theorem as in Section~\ref{sec2} to verify there exists 
   a unique solution $\phi_\varepsilon$ to \eqref{3-2} with $h(\boldsymbol z)=N_\varepsilon(\phi_\varepsilon)$, such that
   \begin{equation*}
   	\|\phi_\varepsilon\|_*+\varepsilon^{1-\frac{2}{p}}\|\nabla\phi_\varepsilon\|_{L^p\left(\left(\cup_{m=1}^jB_{L\varepsilon}(\boldsymbol z_{m,\varepsilon}^+)\right)\cup \left(\cup_{n=1}^kB_{L\varepsilon}(\boldsymbol z_{n,\varepsilon}^-)\right)\right)}= O(\varepsilon)
   \end{equation*}
   for $p\in(2,+\infty]$, provided $\varepsilon\in (0,\varepsilon_0]$ with $\varepsilon_0$ small.
   
	\subsection{The $(2j+2k)$-dimensional problem}
	
	Multiplying \eqref{3-2} by $X_{l,\varepsilon}^\pm$, $Y_{i,\varepsilon}^\pm$ and integrating on $\mathbb S^2$, for the projection vector $\boldsymbol \Lambda=(a_1^+,\cdots,a_j^+,b_1^+,\cdots,b_j^+, a_1^-,\cdots,a_k^-,b_1^-,\cdots,b_k^-)$ we have the equations
	\begin{align*}
		\int_{\mathbb S^2}\left[(-\Delta_{\mathbb S^2})\psi_\varepsilon-\frac{1}{\varepsilon^2}\sum_{m=1}^j\boldsymbol1_{B_\delta(\boldsymbol z_m^+)}\boldsymbol1_{\{\psi_\varepsilon+\gamma\cos\theta>\mu_{m,\varepsilon}^+\}}+\frac{1}{\varepsilon^2}\sum_{n=1}^k\boldsymbol1_{B_\delta(\boldsymbol z_n^-)}\boldsymbol1_{\{-\psi_\varepsilon-\gamma\cos\theta>\mu_{n,\varepsilon}^-\}}\right]X_{m,\varepsilon}^+ d\boldsymbol\sigma\\
		=a_m^+\int_{\mathbb S^2} X_{m,\varepsilon}^+(-\Delta_{\mathbb S^2})X_{m,\varepsilon}^+ d\boldsymbol \sigma,
	\end{align*}
	\begin{align*}
		\int_{\mathbb S^2}\left[(-\Delta_{\mathbb S^2})\psi_\varepsilon-\frac{1}{\varepsilon^2}\sum_{m=1}^j\boldsymbol1_{B_\delta(\boldsymbol z_m^+)}\boldsymbol1_{\{\psi_\varepsilon+\gamma\cos\theta>\mu_{m,\varepsilon}^+\}}+\frac{1}{\varepsilon^2}\sum_{n=1}^k\boldsymbol1_{B_\delta(\boldsymbol z_n^-)}\boldsymbol1_{\{-\psi_\varepsilon-\gamma\cos\theta>\mu_{n,\varepsilon}^-\}}\right]Y_{m,\varepsilon}^+ d\boldsymbol\sigma\\
		=b_m^+\int_{\mathbb S^2} Y_{m,\varepsilon}^+(-\Delta_{\mathbb S^2})Y_{m,\varepsilon}^+ d\boldsymbol \sigma,
	\end{align*}
	\begin{align*}
		\int_{\mathbb S^2}\left[(-\Delta_{\mathbb S^2})\psi_\varepsilon-\frac{1}{\varepsilon^2}\sum_{m=1}^j\boldsymbol1_{B_\delta(\boldsymbol z_m^+)}\boldsymbol1_{\{\psi_\varepsilon+\gamma\cos\theta>\mu_{m,\varepsilon}^+\}}+\frac{1}{\varepsilon^2}\sum_{n=1}^k\boldsymbol1_{B_\delta(\boldsymbol z_n^-)}\boldsymbol1_{\{-\psi_\varepsilon-\gamma\cos\theta>\mu_{n,\varepsilon}^-\}}\right]X_{n,\varepsilon}^- d\boldsymbol\sigma\\
		=a_n^-\int_{\mathbb S^2} X_{n,\varepsilon}^-(-\Delta_{\mathbb S^2})X_{n,\varepsilon}^- d\boldsymbol \sigma,
	\end{align*}
	\begin{align*}
		\int_{\mathbb S^2}\left[(-\Delta_{\mathbb S^2})\psi_\varepsilon-\frac{1}{\varepsilon^2}\sum_{m=1}^j\boldsymbol1_{B_\delta(\boldsymbol z_m^+)}\boldsymbol1_{\{\psi_\varepsilon+\gamma\cos\theta>\mu_{m,\varepsilon}^+\}}+\frac{1}{\varepsilon^2}\sum_{n=1}^k\boldsymbol1_{B_\delta(\boldsymbol z_n^-)}\boldsymbol1_{\{-\psi_\varepsilon-\gamma\cos\theta>\mu_{n,\varepsilon}^-\}}\right]Y_{n,\varepsilon}^- d\boldsymbol\sigma\\
		=b_n^-\int_{\mathbb S^2} Y_{n,\varepsilon}^-(-\Delta_{\mathbb S^2})Y_{n,\varepsilon}^- d\boldsymbol \sigma.
	\end{align*}
	To make $\psi_\varepsilon$ a solution to \eqref{1-8}, the following lemma is established as an analogy of Lemma \ref{lem2-6}.
	\begin{lemma}\label{lem3-2}
		One has
		\begin{align*}
			&\quad\int_{\mathbb S^2}\left[(-\Delta_{\mathbb S^2})\psi_\varepsilon-\frac{1}{\varepsilon^2}\sum_{m=1}^j\boldsymbol1_{B_\delta(\boldsymbol z_m^+)}\boldsymbol1_{\{\psi_\varepsilon+\gamma\cos\theta>\mu_{m,\varepsilon}^+\}}+\frac{1}{\varepsilon^2}\sum_{n=1}^k\boldsymbol1_{B_\delta(\boldsymbol z_n^-)}\boldsymbol1_{\{-\psi_\varepsilon-\gamma\cos\theta>\mu_{n,\varepsilon}^-\}}\right]X_{m,\varepsilon}^+ d\boldsymbol\sigma\\
			&=\kappa_m^+\left[\kappa_i^+\sum\limits_{i\neq m}^j\partial_\theta G(\boldsymbol z_{m,\varepsilon}^+,\boldsymbol z_{i,\varepsilon}^+)+ \kappa_m^+ \partial_\theta H(\boldsymbol z_{m,\varepsilon}^+,\boldsymbol z_{m,\varepsilon}^+)-\kappa_l^-\sum\limits_{l=1}^k\partial_\theta G(\boldsymbol z_{m,\varepsilon}^+,\boldsymbol z_{l,\varepsilon}^-)-\gamma\sin\theta_{m,\varepsilon}^++o_\varepsilon(1)\right],
		\end{align*}
		\begin{align*}
			&\quad\int_{\mathbb S^2}\left[(-\Delta_{\mathbb S^2})\psi_\varepsilon-\frac{1}{\varepsilon^2}\sum_{m=1}^j\boldsymbol1_{B_\delta(\boldsymbol z_m^+)}\boldsymbol1_{\{\psi_\varepsilon+\gamma\cos\theta>\mu_{m,\varepsilon}^+\}}+\frac{1}{\varepsilon^2}\sum_{n=1}^k\boldsymbol1_{B_\delta(\boldsymbol z_n^-)}\boldsymbol1_{\{-\psi_\varepsilon-\gamma\cos\theta>\mu_{n,\varepsilon}^-\}}\right]Y_{m,\varepsilon}^+ d\boldsymbol\sigma\\
			&=\kappa_m^-\left[\kappa_i^+\sum\limits_{i\neq m}^j\partial_\varphi G(\boldsymbol z_{m,\varepsilon}^+,\boldsymbol z_{i,\varepsilon}^+)+ \kappa_m^+ \partial_\varphi H(\boldsymbol z_{m,\varepsilon}^+,\boldsymbol z_{m,\varepsilon}^+)-\kappa_l^-\sum\limits_{l=1}^k\partial_\varphi G(\boldsymbol z_{m,\varepsilon}^+,\boldsymbol z_{l,\varepsilon}^-)+o_\varepsilon(1)\right],
		\end{align*}
		\begin{align*}
			&\quad\int_{\mathbb S^2}\left[(-\Delta_{\mathbb S^2})\psi_\varepsilon-\frac{1}{\varepsilon^2}\sum_{m=1}^j\boldsymbol1_{B_\delta(\boldsymbol z_m^+)}\boldsymbol1_{\{\psi_\varepsilon+\gamma\cos\theta>\mu_{m,\varepsilon}^+\}}+\frac{1}{\varepsilon^2}\sum_{n=1}^k\boldsymbol1_{B_\delta(\boldsymbol z_n^-)}\boldsymbol1_{\{-\psi_\varepsilon-\gamma\cos\theta>\mu_{n,\varepsilon}^-\}}\right]X_{n,\varepsilon}^- d\boldsymbol\sigma\\
			&=\kappa_n^+\left[\kappa_l^+\sum\limits_{l\neq n}^k\partial_\theta G(\boldsymbol z_{n,\varepsilon}^-,\boldsymbol z_{l,\varepsilon}^-)+ \kappa_n^+ \partial_\theta H(\boldsymbol z_{n,\varepsilon}^+,\boldsymbol z_{n,\varepsilon}^+)-\kappa_i^-\sum\limits_{i=1}^j\partial_\theta G(\boldsymbol z_{n,\varepsilon}^+,\boldsymbol z_{i,\varepsilon}^-)-\gamma\sin\theta_{n,\varepsilon}^-+o_\varepsilon(1)\right],
		\end{align*}
		\begin{align*}
			&\quad\int_{\mathbb S^2}\left[(-\Delta_{\mathbb S^2})\psi_\varepsilon-\frac{1}{\varepsilon^2}\sum_{m=1}^j\boldsymbol1_{B_\delta(\boldsymbol z_m^+)}\boldsymbol1_{\{\psi_\varepsilon+\gamma\cos\theta>\mu_{m,\varepsilon}^+\}}+\frac{1}{\varepsilon^2}\sum_{n=1}^k\boldsymbol1_{B_\delta(\boldsymbol z_n^-)}\boldsymbol1_{\{-\psi_\varepsilon-\gamma\cos\theta>\mu_{n,\varepsilon}^-\}}\right]Y_{n,\varepsilon}^- d\boldsymbol\sigma\\
			&=\kappa_n^-\left[\kappa_l^+\sum\limits_{l\neq n}^k\partial_\varphi G(\boldsymbol z_{n,\varepsilon}^-,\boldsymbol z_{l,\varepsilon}^-)+ \kappa_n^+ \partial_\varphi H(\boldsymbol z_{n,\varepsilon}^+,\boldsymbol z_{n,\varepsilon}^+)-\kappa_i^-\sum\limits_{i=1}^j\partial_\varphi G(\boldsymbol z_{n,\varepsilon}^+,\boldsymbol z_{i,\varepsilon}^-)+o_\varepsilon(1)\right].
		\end{align*}
	\end{lemma}
	
    \noindent{\bf Proof of Theorem \ref{thm2}:} To obtain a family of desired solutions to \eqref{1-8}, we need to choose suitable locations $\boldsymbol z_{m,\varepsilon}^+, \boldsymbol z_{n,\varepsilon}^-$ such that $\mathbf \Lambda=\boldsymbol 0$. Notice that 
    \begin{equation*}
    	\int_{\mathbb S^2} X_{l,\varepsilon}^\pm(-\Delta_{\mathbb S^2})X_{l,\varepsilon}^\pm d\boldsymbol \sigma>0 \quad \mathrm{and} \quad \int_{\mathbb S^2} Y_{i,\varepsilon}^\pm(-\Delta_{\mathbb S^2})Y_{i,\varepsilon}^\pm d\boldsymbol \sigma>0.
    \end{equation*} 
    Hence by Lemma \ref{lem3-2} and $(\boldsymbol z_1^+,\cdots,\boldsymbol z_j^+, \boldsymbol z_1^-,\cdots, \boldsymbol z_k^-)$ being a nondegenerate critical point of Kirchhoff--Routh function $\mathcal K_{k+j}$ in \eqref{1-7}, we claim that there exists a proper location series
    \begin{equation*}
    	\left(\boldsymbol z_{1,\varepsilon}^+,\cdots,\boldsymbol z_{j,\varepsilon}^+, \boldsymbol z_{1,\varepsilon}^-,\cdots, \boldsymbol z_{k,\varepsilon}^-\right)= \left(\boldsymbol z_1^+,\cdots,\boldsymbol z_j^+, \boldsymbol z_1^-,\cdots, \boldsymbol z_k^-\right)+o_\varepsilon(1)
    \end{equation*}
    such that $\mathbf \Lambda=\boldsymbol 0$, which gives the existence and limiting behavior for vortex centers 
    $$(\boldsymbol z_{1,\varepsilon}^+,\cdots,\boldsymbol z_{j,\varepsilon}^+, \boldsymbol z_{1,\varepsilon}^-,\cdots, \boldsymbol z_{k,\varepsilon}^-)$$
    in (ii). While the estimates for level sets $\Omega_{m,\varepsilon}^+$ and $\Omega_{n,\varepsilon}^-$ are similar as in Theorem \ref{thm1}. According to our construction, the convergence for $\omega_\varepsilon=(-\Delta_{\mathbb S^2})\psi_\varepsilon$ stated in (i) is obvious. Hence we have completed the proof.	\qed
    
    \bigskip
    
    \noindent{\bf Conflict of interest statement:} On behalf of all authors, the corresponding author states that there is no conflict of interest.
    
    \bigskip
    
    \noindent{\bf Data available statement:} Our manuscript has no associated data.
    
    \bigskip
    
	\phantom{s}
	\thispagestyle{empty}

\end{document}